\documentclass[a4paper,10pt,leqno]{amsart}
	\calclayout
\RequirePackage[T1]{fontenc}
\RequirePackage[utf8]{inputenc}
\RequirePackage{xcolor}
\RequirePackage{amsfonts, amsmath, amssymb, amsthm}
\RequirePackage{etoolbox, mathtools, microtype, stmaryrd, textcomp}
\RequirePackage{hyperref}

\usepackage{xcolor}
\usepackage{comment}
\usepackage{cite}
	\hypersetup{colorlinks}
\RequirePackage{lmodern}

\makeatletter
\g@addto@macro \normalsize {%
 \setlength\abovedisplayskip{5pt plus 2pt minus 2pt}%
 \setlength\belowdisplayskip{5pt plus 2pt minus 2pt}%
}
\makeatother

	\numberwithin{equation}{section}
	\theoremstyle{plain}
		\newtheorem{thm}{Theorem}[section]
	\theoremstyle{plain}
		
		\newtheorem{cor}[thm]{Corollary}
		\newtheorem{lem}[thm]{Lemma}
		\newtheorem{prop}[thm]{Proposition}
	\theoremstyle{definition}
		\newtheorem{df}[thm]{Definition}

	\theoremstyle{remark}
		\newtheorem{rem}[thm]{Remark}

\newcommand{\define}[3]{\expandafter#1\csname#3\endcsname{#2{#3}}}
\forcsvlist{\define{\DeclareMathOperator}{}}{ad,codim,coker,gr,id,im,lcm,ord,pr,rk,sgn,supp,tr}
\forcsvlist{\define{\DeclareMathOperator}{}}{Ad,Aff,Alt,Aut,End,Gal,Heis,Hom,Ind,Inf,Irr,Mat,Sym}
\forcsvlist{\define{\DeclareMathOperator}{}}{GL,PGL,PSL,SL,SO,SU}
\forcsvlist{\define{\newcommand}{\mathrm}}{ab,op,sep,tor}

\newcommand{\T}{\mathbb{T}}
\newcommand{\Z}{\mathbb{Z}}
\newcommand{\R}{\mathbb{R}}
\newcommand{\C}{\mathbb{C}}
\newcommand{\CC}
{\mathcal{C}}
\newcommand{\E}{\mathbb{E}}
\newcommand{\N}{\mathbb{N}}
\newcommand{\F}{\mathcal{F}}
\newcommand{\Q}{\mathcal{Q}}

\newcommand{\dd}{\,\text{d}}

\newcommand{\FF}{\mathcal{F}}

\title{Recent results On The modulated cubic Nonlinear Schr\"{o}dinger equation on $\T^2$}
\author{Josh Messing}
\address{Department of Mathematics, Massachusetts Institute of Technology}
\email{jemess@mit.edu}

\begin{document}
\begin{abstract}
New Strichartz estimates for the modulated cubic nonlinear Schr\"{o}dinger equation are proved. These Strichartz estimates allow us to show that this equation is pathwise locally well-posed. We also show that improved Strichartz estimates are available in the case where the modulation is white noise. Additionally, we comment on a few basic properties of the modulated cubic nonlinear Schr\"{o}dinger equation such as conservation of mass and convergence of its linear flow as time tends to zero.
\end{abstract}
\thanks{This work was funded in part by the NSF grant DMS-2306378 and the Simons Foundation through the Simons Collaboration on Wave Turbulence.}
\maketitle

\section{Introduction}\label{Sec:Introduction}
In this paper we study the modulated cubic nonlinear Schr\"{o}dinger equation
\begin{equation}
\begin{cases}
\frac{du(t,x)}{dt} + \frac{dW_t}{dt}\Delta u(t,x) = |u(t,x)|^2u(t,x)\\
u(0,x) = u_0 \in H^s(\T^2)
\end{cases}
\end{equation}
where $\T^2 $ is the standard square torus. The function $W_t$, which we refer to as the modulation, is continuous but no additional regularity is assumed so we interpret this equation in its Duhamel form
\begin{equation}\label{Eq:ModulatedNLSDuhamel}
u(t,x) = e^{iW_t\Delta}u_0(x) - i\int_0^t e^{i(W_t-W_{t'})\Delta}|u(t',x)|^2u(t',x) \dd t'. \tag{MNLS}
\end{equation}
When $W_t = t$ this is simply the cubic nonlinear Schr\"{o}dinger equation (NLS). Bourgain pioneered the study of well-posedness, namely existence, uniqueness, and stability with respect to initial data, for this type of equation in his seminal paper \cite{bourgain93}. Later Bourgain and Demeter's proof of the $\ell^2$ decoupling conjecture expanded well-posedness results to irrational tori \cite{bourgain2015proofl2decouplingconjecture} where the ratio between the two periods of the torus is irrational. NLS equations have been studied extensively but MNLS equations are not as well understood. The case where $W_t$ is a path of a standard Brownian motion is of particular interest since it has applications to settings with random dispersion and optical fibers \cite{DEBOUARD20101300, debussche2010quinticNLS}. In this case this equation is referred to as the cubic NLS with white noise dispersion and there are already a number of results on its well-posedness theory. In the one dimensional nonperiodic case Debushe et al. use stochastic Strichartz estimates to perform contraction arguments in stochastic Strichartz spaces \cite{DEBOUARD20101300, debussche2010quinticNLS}. Stewart uses a similar approach to study the one dimensional periodic NLS with white noise dispersion \cite{stew}. While these methods have been fruitful for obtaining probabilistic\footnote{A probabilistic result typically holds for data that is measured in some $L^p(\Omega)$ based space where $\Omega$ is the probability space on which $W_t$ is defined. In this paper we contrast this with pathwise results that hold on paths $t\mapsto W_t$.} local and global well-posedness results on $\R$ and $\T$ the author does not know of any well-posedness results for a MNLS, white noise or otherwise, in the higher dimensional periodic setting. 

An alternative approach has been developed by Chouk and Gubinelli in \cite{chouk2015nonlinearpdesmodulateddispersion}. Unlike the aformentioned results their work does not rely on $W_t$ being a standard Brownian motion. Instead they use ideas from the theory of rough paths to construct a solution to MNLS as long as the modulation satisfies an irregularity condition which we discuss in Section \ref{Sec:StrichartzEstimates} and define in Definition \ref{Defn:RhoGammaIrregular} for convenience. They use the fact that a standard Brownian motion almost surely satisfies this irregularity condition to construct solutions to MNLS on almost every path $t \mapsto W_t(\omega)$ where $\omega$ is sampled from the probability space $(\Omega, \FF, P)$ on which the standard Brownian motion is defined. We refer to this type of result as pathwise well-posedness. It is weaker in spirit than the well-posedness results described above. The estimates that are obtained hold almost surely in $\Omega$. On the other hand the aformentioned stochastic Strichartz estimates that are proven in \cite{DEBOUARD20101300, debussche2010quinticNLS, stew} all hold in an $L^p(\Omega)$ based spaces. 

Much of this paper is devoted to the study of pathwise local well-posedness for MNLS on $\T^2$. We combine Chouk and Gubinelli's approach with ideas from a recent paper of Herr and Kwak \cite{IncidenceGeometryStrichartz}. Herr and Kwak prove global well-posedness of the cubic NLS on $\T^2$ for small initial data. They later expanded on this work in \cite{herr2025globalwellposedness} to remove the assumption of small initial data. One of the key contributions of their work is an application of incidence geometry to obtain a sharp Strichartz estimate for the Schr\"{o}dinger propagator $\{e^{it\Delta}\}_{t\geq0}$. They show that there is a constant $M$ such that for all $f \in L^2(\T^2)$, $N \in \N$, $C \subset \Z^2$ such that $|C|=N$, and intervals $I \subset [0, \infty)$ with length $|I| \leq \frac{1}{\log N}$, 
\begin{equation}
\|e^{it\Delta}P_Cf(x)\|_{L^4(I \times \T^2)} \leq M\|f(x)\|_{L^2(\T^2)},   
\end{equation}
where $P_C$ is the Fourier projection onto the set $C$. 

In Theorem \ref{Thm:PathwiseStrichartz} we combining the approaches of Herr and Kwak as well as Chouk and Gubinelli to prove that a similar Strichartz type estimate holds for the modulated Schr\"{o}dinger propagator $\{e^{i W_t \Delta}\}_{t\geq 0}$ as long as $W_t$ is sufficiently irregular. Here we give an informal statement of this result which will be stated more precisely in Section \ref{Sec:StrichartzEstimates}.

\begin{thm}[Informal Pathwsie Strichartz Estimates]\label{T11}
Let $\epsilon > 0$ and $W_t:[0,\infty) \rightarrow \R$ be a function with sufficient "irregularity" as discussed in Section \ref{Sec:StrichartzEstimates}. For all rectangles $S = \{(\xi_1, \xi_2): |\xi_1-a|,|\xi_2-b| \leq N\} \subset \Z^2$, intervals $I \subset [0,1]$ with $|I|=T_0$, and $f \in L^2(\T^2)$ we have \[
 \|e^{iW_t\Delta}P_Sf\|_{L^4(I \times \T^2)} \lesssim_{W_t, \epsilon} (N^{2\epsilon}\sqrt{T_0})^{1/4}\|f\|_{L^2(\T^2)}.
 \]  
 Moreover, in the case $W_t =B^H_t$ is a fractional Brownian motion with Hurst index $H \leq \frac{1}{2}$ then a stronger "Stochastic Strichartz" estimate holds.
\end{thm}

In Section \ref{Sec:StrichartzEstimates} we discuss how Theorem \ref{T11} can come in part from leveraging the decay in $\tau$ of oscillatory integral expressions of the form $|\Phi^W_t(\tau) - \Phi^W_s(\tau)|$ where $\Phi^W_T(\tau) := \int_0^T e^{iW_t \tau} \dd t$. The decay in $\tau$ of this expression turns out to be the exact type of irregularity condition that Chouk and Gubinelli study in \cite{chouk2015nonlinearpdesmodulateddispersion}. 

Going back to Theorem \ref{T11}, we note that the propagator $\{e^{iW_t\Delta}\}_{t\geq 0}$ solves the linear part of \eqref{Eq:ModulatedNLSDuhamel} so these Strichartz estimates play an important role in the proof of pathwise local well-posedness. In Theorem \ref{Thm:LocalWellPosedness} and Corollary \ref{Cor:PathWiseLocalWell-posednessBrownian} we prove this for initial data in $H^s$ for all $s>0$. Here we state the result informally.  

\begin{thm}[Informal Local Well-posedness]\label{T12}
If $W_t$ is sufficiently "irregular" then \eqref{Eq:ModulatedNLSDuhamel} is locally well-posed for initial data in $H^s(\T^2)$ for $s > 0$. Moreover, if $W_t = B^H_t$ is a fractional Brownian motion with Hurst index $H \leq \frac{1}{2}$ then the paths of $W_t$ are almost surely "irregular" enough so that \eqref{Eq:ModulatedNLSDuhamel} is pathwise locally well-posed. 
\end{thm}

One downfall of this pathwise local well-posedness approach is that if $W_t$ is a standard Brownian motion on $(\Omega, \FF, P)$ then the time of existence for \eqref{Eq:ModulatedNLSDuhamel} is not guaranteed to be uniform in $\Omega$ even if the initial data is uniformly small. As proved by Herr and Kwak in \cite{herr2025globalwellposedness} the deterministic cubic NLS with $W_t=t$ is globally well-posed for initial data in $H^s(\T^2)$ for all $s>0$. One may hope that a similar result could hold for \eqref{Eq:ModulatedNLSDuhamel} when $W_t$ is a standard Brownian motion. Stewart obtains such a result in the one dimensional case in \cite{stew}. However, their methods are not available to us because Strichartz estimates are much weaker in higher dimensions. 

In Theorem \ref{T11} or Theorem \ref{Thm:PathwiseStrichartz} we prove a stochastic Strichartz type result that holds in an $L^4(\Omega)$-based space, where $(\Omega, \FF, P)$ is a probability space. This estimate is as strong in the time and space variables as the sharp result obtained by Herr and Kwak in \cite[Theorem 1.2]{IncidenceGeometryStrichartz}. This suggest the possibility of obtaining some type of local or global well-posedness result that holds in some $L^4(\Omega)$ based space. We attempted to use this estimate to perform contraction arguments in stochastic versions of Bourgain's $X^{s,b}$ spaces, and in stochastic versions of $Y^s(V^2)$ type spaces (defined in Definition \ref{Defn:YSpace}), which have been used for the study of Schr\"{o}dinger equations at critical regularity \cite{IncidenceGeometryStrichartz, herr2025globalwellposedness, XsYsEstimates}. In both approaches we came upon a number of technical difficulties relating to controlling integration in $\Omega$, and we expect to continue our investigation in this direction in the future. 

Because \eqref{Eq:ModulatedNLSDuhamel} has not been studied on $\T^2$ we include a few other results about its general properties in Section \ref{Sec:AdditionalResults}. We show that the $L^2(\T^2)$ mass of solutions to \eqref{Eq:ModulatedNLSDuhamel} is conserved in time. This property is interesting in its own right and is often necessary for global well-posedness arguments that rely on iterating a local result with the time of existence depending on the $L^2(\T^2)$ mass of the initial data. We also study the problem of convergence of the linear modulated Schr\"odinger flow to initial data. This measure of continuity of the solution was first studied by Carlson in \cite{carleson1980} and has since been studied extensively in many different contexts. See \cite{MoyuaVega2008DEquals1SFlow,pointwiseconvergenceschrodingerflow,du2017sharpschrodingermaximalestimate, du2019sharpl2estimateschrodinger} and references therein.

Finally, we note that Theorem \ref{T11} and Theorem \ref{T12} could easily be extended to any "rational" torus $\R^2 / \alpha \Z \times \beta \Z$ where $\alpha $ and $\beta$ are a rational multiple of each other. At the moment it is not clear how to extend our results to an irrational torus $\R^2 / \alpha \Z \times \beta \Z$ where $\alpha$ and $\beta$ are not a rational multiple of each other. For the standard linear Schr\"odinger operator ($W_t=t$) many results related to the Strichartz estimates on irrational tori have been proved via Bourgain and Demeter's proof of the $\ell^2$ decoupling conjecture in \cite{bourgain2015proofl2decouplingconjecture}. It is not clear if these ideas have analogues in the setting where $W_t$ is a more exotic modulation.

\textbf{Outline of the paper:} In Section \ref{Sec:Preliminaries} we fix the notation that will be used in the rest of the paper. Section \ref{Sec:StrichartzEstimates} contains the proof of Strichartz estimates for the modulated Schr\"{o}dinger propagator $\{e^{iW_t\Delta}\}_{t\geq 0}$, as well as an extensive discussion of the methods that are used. In Section \ref{Sec:FunctionSpaces} we discuss the function spaces on which we will perform a Picard iteration scheme to derive solutions to \eqref{Eq:ModulatedNLSDuhamel}. In Section \ref{Sec:Trilinear} we prove a number of key estimates that are used for the proof of pathwise local well-posedness in Theorem \ref{Thm:LocalWellPosedness} in Section \ref{Sec:Local}. In Section \ref{Sec:AdditionalResults} we study conservation of mass for \eqref{Eq:ModulatedNLSDuhamel} and convergence of the propagator $\{e^{iW_t\Delta}\}_{t \geq 0}$ as $t$ tends to zero when $W_t$ is a stochastic process satisfying an integrability condition. 
\section{Preliminaries}\label{Sec:Preliminaries}
Throughout this paper we will use $|S|$ to denote the cardinality of a set $S$. If $I \subset \R$ is an interval we will use $|I|$ for the length of the interval. the notation $A \lesssim B $ or $A \gtrsim B$ will mean that there is some constant $C$ such that one always has $A \leq CB $ or $A \geq CB $ respectively. $A \sim B$ will mean that $A \lesssim B$ and $B \lesssim A$. The constant $C$ may depend on choices that are held constant throughout the construction in question. For example, it may depend on the choice of modulation $W_t$, the Sobolev space $H^s(\T^2)$ where the initial data lies, or geometric constants coming from the fact that we work exclusively on the square torus $\T^2 := \R^2 / 2\pi\Z^2$. If we want to be explicit about dependencies we may write $A \lesssim_{\alpha} B$ or $B \gtrsim_{\alpha}C$ to denote that $A \lesssim B$ or $A \gtrsim B$ with implicit constants depending on a parameter $\alpha$. If $f \in L^2(\T^2)$ will denote the k'th coefficient of the complex fourier series with $f_k$
\begin{equation*}
f_k = \int_{\T^2} e^{-i k x}f(x) \dd x
\end{equation*}
where $kx$ is shorthand for $k \cdot x$. For a set $C \subset \Z^2$ we define the Fourier projection multiplier
\begin{equation*}
P_Cf(x) = \sum_{k \in C}f_ke^{ik\cdot x}.
\end{equation*}
If $M \geq 0$ then $P_{\gtrsim M}$ denotes projection onto frequencies $k$ where $|k|\gtrsim M$ and $P_{\sim M} $ denotes projection onto frequencies $k$ where $|k| \sim M$.

\section{Strichartz Estimates}\label{Sec:StrichartzEstimates}
Before we state and prove Strichartz estimates we will motivate our approach, which is heavily inspired by Herr and Kwak's proof of Strichartz estimates for NLS on $\T^2$ in \cite[Theorem 1.2]{IncidenceGeometryStrichartz}. Along the way we will define some notation and highlight some key lemmas that are used in the proof. First we discuss the relationship between oscillatory integrals of the form $\int_0^T e^{iW_t \tau} \dd t$ and the Strichartz estimates we prove in Theorem \ref{Thm:PathwiseStrichartz}.  We rewrite $\|e^{iW_t\Delta}f\|_{L^4(\T^2)}$ in terms of the coefficients $f_k$ of the Fourier series of $f$.
\begin{align}
&\int_{0}^T \int_{\T^2} |e^{-iW_t\Delta}f|^4 \dd x \dd t \notag\\
&=
\sum_{k_1-k_2+k_3-k_4=0}f_{k_1}\overline{f_{k_2}}f_{k_3}\overline{f_{k_4}}\int_{0}^T e^{iW_t(|k_1|^2-|k_2|^2+|k_3|^2-|k_4|^2)} \dd t. \label{Intermediary21}
\end{align}
Let $\tau = |k_1|^2-|k_2|^2+|k_3|^2-|k_4|^2 $.  In the case where $W_t=t$ the integral of $e^{it\tau}$ with respect to $t$ vanishes on the interval $[0,2\pi]$ except when $\tau=0$. In this case one can estimate \eqref{Intermediary21} by counting the number of $4$-tuples $(k_1,k_2,k_3,k_4)$ of lattice points where $k_1-k_2+k_3-k_4=0$ and $\tau=0$. In Bourgain's seminal paper \cite{bourgain93} he uses the "circle method" to do this and obtain Strichartz estimates for the periodic Schr\"{o}dinger propagator $\{e^{it\Delta}\}_{t=0}$. When $W_t$ is a more general modulation this approach is not available. Instead of using that the integral of $e^{it\tau}$ vanishes when $\tau \neq 0$ we use that the integral of $e^{iW_t\tau}$ is small when $\tau$ is large. First, we group terms on the basis of the value of $\tau$ and define some notation to be used later.
 
\begin{df}\label{Defn:QTau}
Define $\mathcal{Q}$ to be
\begin{equation*}
\mathcal{Q}:=\left\{Q = (k_1, k_2, k_3, k_4)\in \Z^{2 \times 4} \,:\, k_1-k_2+k_3-k_4=0\right\}.   
\end{equation*}
\end{df}
The elements of $\mathcal{Q}$ are ordered $4$-tuples of vertices of parallelograms in the integer lattice. For this reason we will refer to them as parallelograms. For such a parallelogram we set $\tau_{Q}:= |k_1|^2-|k_2|^2+|k_3|^2-|k_4|^2 $ and $Q^{\tau}:= \{Q \in \mathcal{Q}: \tau_{Q}=\tau\}$. If $f$ is a function on the integer lattice, for any parallelogram $Q=(k_1, k_2, k_3, k_4)$ we set $f(Q) = f(k_1)\overline{f(k_2)}f(k_3)\overline{f(k_4)}$. 

\begin{rem}\label{Rem:TauInvarianceUnderTranslation}
The quantity $\tau_{Q}$ is invariant under translations of the  parallelogram $Q$. In particular if we are considering only parallelograms with vertices in some given set $C$, the set $\{\tau_{Q}:Q \subset C\}$ is invariant under translations of $C$.
\end{rem}
Using this notation we can rewrite \eqref{Intermediary21} as
\begin{align}\label{Exp:MotivationForIrregularity}
\sum_{\tau \in \Z} \sum_{Q \in Q^{\tau}} \hat{f}(Q) \int_0^T e^{iW_t\tau}\dd t.
\end{align}
In order to bound this quantity one needs to control $\int_0^T e^{iW_t \tau} \dd t$ and $\sum_{Q \in Q^{\tau}} \hat{f}(Q)$. To handle the latter we use the  following two propositions, the first of which is a result of Herr and Kwak, and the second of which is a standard corollary of the Szemeredi-Trotter theorem.
\begin{prop}\cite[Proposition 3.1]{IncidenceGeometryStrichartz}\label{Prop:PropositionForStrichartz}

Let $f: \Z^2 \rightarrow [0, \infty) $ be a function of the form
\[
f = \sum_{j=0}^m \lambda_j 2^{-j/2} 1_{S_j}
\]
where $S_0, \ldots, S_m$, $m \geq 1$ are disjoint subsets of $\Z^2$ such that $\big{|}S_j\big{|} \leq 2^j$, and $\lambda_0, \ldots, \lambda_m \geq 0$. Suppose that for each $j = 0, \ldots, m$ and $\xi \in S_j$, there exists at most one line $\ell \supset \xi $ such that $\big{|}(\ell \cap S_j)\big{|} \geq 2^{j/2 + C}$ for some fixed $C$. Then, we have
\begin{equation*}
    \sum_{Q \in \Q^0} f(Q) \lesssim m \cdot \|\lambda_j\|_{\ell^2_{j \leq m}}^4 := m\left(\sum_{j \leq m}|\lambda_j|^2\right)^2,
\end{equation*}
and
\begin{equation*}
\sup_{M \in 2^{\N}} \frac{1}{M} \sum_{ \tau \approx M}\sum_{Q \in \Q^{\tau}}f(Q) \lesssim \|\lambda_j\|_{\ell^2_{j \leq m}}^4 := \left(\sum_{j \leq m}|\lambda_j|^2\right)^2.
\end{equation*}
\end{prop}

\begin{prop}\cite[Corollary 8.5]{tao-vu-additive}
\label{prop:SzTr}Let $S\subset\R^{2}$ be a set of $n$ points,
where $n\in\N$. Let $k\geq2$ be an integer. The number $m$ of lines
in $\R^{2}$ passing through at least $k$ points of $S$ satisfies the estimate 
\begin{equation*}
m\lesssim\frac{n^{2}}{k^{3}}+\frac{n}{k}.\label{eq:SzTr}
\end{equation*}
\end{prop}

The other factor in \eqref{Exp:MotivationForIrregularity} is the oscillatory integral $\int_0^T e^{iW_t\tau} \dd t$. The decay properties of this expression have been formalized in \cite{chouk2015nonlinearpdesmodulateddispersion,catellier2016averagingirregularcurves, galeati2023prevalencerhoirregularity}  with the notion of $(\rho, \gamma)$-irregularity.
\begin{df}\label{Defn:RhoGammaIrregular}
Let $W_t \in C^0([0,T], \R)$. $W_t$ is called $(\rho, \gamma)$ irregular if
\begin{equation*}
\sup_{s<t}\sup_{\tau \in \R}\frac{(1+|\tau|)^{\rho}|\int_s^t e^{iW_u\tau} \dd u|}{|s-t|^{\gamma}} < \infty.
\end{equation*}
\end{df}
The Strichartz estimates we prove hold in the case when $W_t$ is $(\rho, \gamma)$-irregular for all $\rho < 1$ and $\gamma = \frac{1}{2}$. The relevance of these values is highlighted by Theorem 1.4 in \cite{catellier2016averagingirregularcurves}, which we report below for completeness.
\begin{thm}\label{Thm:BrownianIsRhoIrregular}
Let $(B^H_t)_{t > 0}$ be a fractional Brownian motion of Hurst index $H \in (0,1)$. Then for any $\rho < 1/2H$ there exist $\gamma > 1/2$ so that with probability one the sample paths of $B^H$ are $(\rho, \gamma)$ irregular.
\end{thm}
It follows that the sample paths of $B^H_t$ are $(\rho, \gamma)$ irregular for all $\rho < 1$ when $H \leq 1/2$. In particular this includes the case of a standard Brownian motion where $H = 1/2$. Note that if $W_t$ is $(\rho, \gamma)$ irregular for $\gamma > 1/2 $ then it is $(\rho, 1/2)$-irregular when restricted to bounded intervals. For example, a fractional Brownian motion with Hurst index $H \leq 1/2$ is $(\rho, 1/2)$-irregular for all $\rho < 1$. 

With this notation in mind we prove Strichartz estimates for the propagator $\{e^{iW_t\Delta}\}_{t\geq 0}$ on $\T^2$. As mentioned in the beginning of this section, the proof is inspired Herr and Kwak's use of incidence geometry for Strichartz estimates for the cubic NLS on $\T^2$.

\begin{thm}\label{Thm:PathwiseStrichartz} The following two types of Strichartz estimates hold
\begin{enumerate}
    \item[1][\textbf{Pathwise Strichartz Estimates}] Let $\epsilon > 0$ and $W_t:[0,\infty) \rightarrow \R$ be a $(\rho, 1/2)$-irregular function as in Definition \ref{Defn:RhoGammaIrregular}. For all rectangles $S = \{(\xi_1, \xi_2): |\xi_1-a|,|\xi_2-b| \leq N\} \subset \Z^2$, intervals $I$ such that $|I| =T_0$, and $f \in L^2(\T^2)$ we have \[
 \|e^{iW_t\Delta}P_Sf\|_{L^4(I \times \T^2)} \lesssim_{W_t, \epsilon} (N^{2\epsilon}\sqrt{T_0})^{1/4}\|f\|_{L^2(\T^2)}.
 \]

    \item[2][\textbf{Stochastic Strichartz Estimates}] Let $W_t = B^H_t$ be a fractional Brownian motion with Hurst index $H \leq \frac{1}{2}$ defined on a probability space $(\Omega, \FF, P)$. For all rectangles $S = \{(\xi_1, \xi_2): |\xi_1-a|,|\xi_2-b| \leq N\} \subset \Z^2$, intervals $I$ such that $|I| =T_0$, and $f \in L^2(\Omega \times \T^2)$ with Fourier coefficients that are independent of the Brownian increments we have
\[
 \|e^{iW_t\Delta}P_Sf\|_{L^4(\Omega \times I \times \T^2)} \lesssim_H (\log(N)T_0 + T_0^{1-H})^{1/4}\|f\|_{L^4(\Omega)L^2(\T^2)}.
\]
\end{enumerate} 
\end{thm}

\begin{rem}
Notice that the stochastic Strichartz estimate is much stronger than the pathwise Strichartz estimate. This comes from the fact that the pathwise Strichartz estimate utilize the decay in $\tau$ of expressions of the form $\int_0^{T_0} e^{-iW_t \tau} \dd t$ while the stochastic Strichartz estimate utilizes the decay in $\tau $ of the expectation of the integral $\int_0^{T_0} E[e^{-iB^H_t \tau}] \dd t = \int_0^{T_0} e^{-\frac{\tau^2}{2}t^{2H}} \dd t $. Replacing $\tau$ with $\tau^2$ makes this integral much smaller for larger $\tau$. 
\end{rem}

\begin{proof} 
It suffices to show that for $f:\Z^{2}\rightarrow \C$ supported
in $S$,
\begin{equation*}
\|e^{iW_t\Delta}\F^{-1}f\|_{L_{t,x}^{4}([0,T]\times\T^{2})}\lesssim (N^{2\epsilon}\sqrt{T_0})^{1/4}\|f\|_{\ell^{2}(\Z^{2})}.
\end{equation*}
The function $f$ can be separated into its real and imaginary parts. Then we can restrict these functions to sets where they are nonpositive or nonnegative. As a result we may assume without loss of generality that $f$ is a nonnegative real valued function. This will allow us to group together values of $\xi \in \Z^2$ where $f$ is large or small. We define a sequence $\left\{ f_{n}\right\}_{n=0}^{n=\infty} $ of functions $f_{n}:\Z^{2}\rightarrow[0,\infty)$ such that $\supp(f_{n})\subset S$. From this sequence we will build $\{g_n\}$ and $\{h_n\}$, two additional auxiliary sequences of functions. Let $f_{0}:=f$. Given $n\in\N$ and a function $f_{n}$ we describe a process that can be used to construct $f_{n+1}$. We choose an enumeration $\xi_{1},\xi_{2},\ldots$ in $S \subset \Z^{2}$ (which
may depend on $n$) such that $f_{n}(\xi_{1})\ge f_{n}(\xi_{2})\ge\ldots$.
Let $S_{j}^{0}:=\left\{ \xi_{2^{j}},\ldots,\xi_{2^{j+1}-1}\right\} $
and $\lambda_{j}:=2^{j/2}f_n(\xi_{2^{j}})$ for $j=0,\ldots,m \approx \log(N^2)$. We have
\begin{equation}\label{Intermediary24}
|S_{j}^{0}|= |\{\xi_{2^{j}}, \ldots, \xi_{2^{j+1}-1}\}| =2^{j},
\end{equation}
and
\begin{align}
\|\lambda_{j}\|_{\ell_{j\leq m}^{2}}
&=
\|2^{j/2}f_{n}(\xi_{2^{j}})\|_{\ell_{j\le m}^{2}}\label{Intermediary13} \\
&\lesssim
f_{n}(\xi_{1})+\|\sum_{j=1}^{m}f_{n}(\xi_{2^{j}})1_{\left\{ \xi_{2^{j}},\ldots,\xi_{2^{j}-1}\right\} }\|_{\ell^{2}(\Z^{2})}\nonumber 
\\
&\lesssim 
\|f_{n}\|_{\ell^{2}(\Z^{2})}. \nonumber 
\end{align}

Let $C$ be a positive constant to be chosen later. For $j=0,\ldots,m$, we define $E_{j}\subset S_{j}^{0}$ as the set
of intersections $\xi\in S_{j}^{0}$ of two lines $\ell_{1},\ell_{2}$
such that
\[
\big{|}\ell_{1}\cap S_{j}^{0}\big{|}, \big{|}\ell_{2}\cap S_{j}^{0} \big{|}
\geq 2^{j/2+C}.
\]
By Proposition \ref{prop:SzTr} we have
\begin{align}
\sqrt{|E_{j}|} & \leq 
\bigg{|}\{ \ell\subset\R^{2}:\ell\text{ is a line and }\big{|}\ell\cap S_{j}^{0}\big{|} \geq 2^{j/2+C}\} \bigg{|} \label{eq:E_j bound} \\
& \lesssim (\big{|}S_{j}^{0}\big{|})^{2}/(2^{j/2+C})^3+\big{|}S_{j}^{0}\big{|}/2^{j/2+C}\nonumber \\
 & \lesssim 2^{j/2-C}.\nonumber 
\end{align}

Let $f_{n+1}:\Z^{2}\rightarrow[0,\infty)$ be the function
\[
f_{n+1}:=f_{n}1_{E}, \;\; \text{where} \;\; E:=\bigcup_{j=0}^{m}E_{j}.
\]
Since $f_{n}(\xi)\le f_{n}(\xi_{2^{j}})=\lambda_{j}2^{-j/2}$ holds for
$\xi\in E_{j}\subset S_{j}^{0}$, by \eqref{eq:E_j bound} and \eqref{Intermediary13},
we have
\[
\|f_{n+1}\|_{\ell^{2}(\Z^{2})}=\|f_{n}{1_{E}\|_{\ell^{2}(\Z^{2})}}\lesssim \|\lambda_{j}2^{-j/2}\cdot\sqrt{\big{|}E_{j}}\big{|}\|_{\ell_{j\le m}^{2}}\lesssim 2^{-C}\|f_{n}\|_{\ell^{2}(\Z^{2})}.
\]
Choose $C \in \N$ so that
\[
\|f_{n+1}\|_{\ell^{2}(\Z^{2})}\le\frac{1}{2}\|f_{n}\|_{\ell^{2}(\Z^{2})},
\]
which implies
\begin{equation}
\|f_{n}\|_{\ell^{2}(\Z^{2})}\le\frac{1}{2}\|f_{n-1}\|_{\ell^{2}(\Z^{2})}\leq \ldots\le2^{-n} \|f\|_{\ell^{2}(\Z^{2})}.\label{eq:f_n _2 < 2^-n f_2}
\end{equation}

Let $S_{j}:=S_{j}^{0} \setminus E_{j}$. By the definition of $E_{j}$,
the function
\[
g_{n}:=\sum_{j=0}^{m}\lambda_{j}2^{-j/2}1_{S_{j}}
\]
satisfies the conditions for Proposition \ref{Prop:PropositionForStrichartz}.
Since $f_n(\xi)\leq f_n(\xi_{2^{j}})=\lambda_{j}2^{-j/2}$ holds for $\xi\in S_{j}\subset S_{j}^{0}=\left\{ \xi_{2^{j}},\ldots,\xi_{2^{j+1}-1}\right\} $ and $f_{n+1} := f_n|_{E}$,
we have
\begin{equation*}
h_{n}:=f_{n}-f_{n+1}=\sum_{j=0}^{m}f_{n}1_{S_{j}}\leq\sum_{j=0}^{m}\lambda_{j}2^{-j/2}1_{S_{j}}=g_{n}.\label{eq:h_n < g_n}
\end{equation*}

Now we use the $(1-\epsilon, 1/2)$-irregularity (see Definition \ref{Defn:RhoGammaIrregular}) of $W_t$ to deduce Strichartz estimates for $h_n$ for each $n \geq 0$. Then we will use those to prove Strichartz estimates  for $f_n$. Recalling Definition \ref{Defn:QTau},
\begin{align} 
&\int_{0}^{T_{0}}\int_{\T^{2}}\bigg{|}e^{iW_t\Delta}(\F^{-1}h_{n})(x)\bigg{|}^{4}\dd x \dd t \label{Intermediary28} \\
& = 
\int_{0}^{T_{0}}\F(|e^{iW_t\Delta}\F^{-1}h_{n}|^{4})(0)\dd t \notag \\
& = 
  \sum_{Q\in \Q}h_{n}(Q)\cdot\int_{0}^{T_{0}} e^{-iW_t\tau_{Q}}\dd t  \notag\\
& =
  \sum_{Q \in \mathcal{Q}^0}h_n(Q)T_0 + \sum_{\tau \neq 0}\sum_{Q \in \mathcal{Q}^{\tau}}h_n(Q) \int_0^{T_0} e^{-iW_t \tau} \dd t \notag\\
&\lesssim 
\sum_{Q \in \mathcal{Q}^0}h_n(Q)T_0 + \sum_{\tau > 0}\frac{\sqrt{T_0}}{\tau^{1-\epsilon}} \sum_{Q \in \mathcal{Q}^{\tau}}h_n(Q).\label{Intermediary23}
\end{align}
Note that we used the nonnegativity of $h_n$ in \eqref{Intermediary23}. Recall from Remark \ref{Rem:TauInvarianceUnderTranslation} that the set $\{\tau_{Q}: Q \subset S\} $ is invariant under translations. Since $S$ is a square with side length $2N$ we may assume without loss of generality that $S = \{(\xi_1, \xi_2): |\xi_1|,|\xi_2| \leq N\}$. For $Q = (k_1, k_2, k_3, k_4)$ we have $\tau_Q = |k_1|^2-|k_2|^2+|k_3|^2-|k_4|^2$ and therefore $\sup\{|\tau_Q|:Q \subset S\} \lesssim N^2$. If we sum up to $\tau = N^2$ in \eqref{Intermediary23}  and split the sum in $\tau$ into dyadic blocks we obtain
\begin{align*}
&\sum_{Q \in \mathcal{Q}^0}h_n(Q)T_0 + \sum_{\tau > 0}\frac{\sqrt{T_0}}{\tau^{1-\epsilon}} \sum_{Q \in \mathcal{Q}^{\tau}}h_n(Q) \\
&\lesssim
\sum_{Q \in \mathcal{Q}^0}h_n(Q)T_0 + \sum_{\ell =1}^{\log(N^2)} \frac{\sqrt{T_0}}{2^{\ell(1-\epsilon)}} \sum_{\tau=2^{\ell-1}}^{2^\ell}  \sum_{Q \in \mathcal{Q}^{\tau}}h_n(Q) \\
&\lesssim
\sum_{Q \in \mathcal{Q}^0}g_n(Q)T_0 + \sum_{\ell =1}^{\log(N^2)} \frac{\sqrt{T_0}}{2^{\ell(1-\epsilon)}} \sum_{\tau=2^{\ell-1}}^{2^{\ell}}  \sum_{Q \in \mathcal{Q}^{\tau}}g_n(Q). 
\end{align*}
Here, and throughout this proof, we use the general convention that if $\gamma \in [0, \infty)$ is not an integer then an expression such as $\sum_{\ell=1}^{\gamma}a_{\ell}$ is shorthand for $\sum_{\ell=1}^{\lceil \gamma \rceil }a_{\ell}$ and $\sum_{\ell=\gamma}^N a_{\ell}$ is short for $\sum_{\ell=\lfloor \gamma \rfloor}^N a_{\ell}$. Using Proposition \ref{Prop:PropositionForStrichartz} we bound this, up to a constant factor, by
\begin{align}
T_0\log(N)\|\lambda_j\|_{\ell^2}^4 + \sum_{\ell=1}^{\log(N^2)} \frac{\sqrt{T_0}}{2^{\ell(1-\epsilon)}}2^{\ell-1}\|\lambda_j\|_{\ell^2}^4 \nonumber 
&\lesssim 
T_0\log(N)\|\lambda_j\|_{\ell^2}^4 + \sum_{\ell=1}^{\log(N^2)} 2^{\ell\epsilon}\sqrt{T_0}\|\lambda_j\|_{\ell^2}^4 \nonumber \\
&\lesssim 
T_0\log(N)\|\lambda_j\|_{\ell^2}^4 + \frac{N^{2\epsilon}-2^{\epsilon}}{2^{\epsilon}-1}\sqrt{T_0}\|\lambda_j\|_{\ell^2}^4 \nonumber \\
&\lesssim_{\epsilon} 
(T_0\log(N) + N^{2\epsilon}\sqrt{T_0})\|\lambda_j\|_{\ell^2}^4. \label{Intermediary29}
\end{align}
Equations \eqref{Intermediary24}, \eqref{Intermediary13}, and the definition of $g_n$ imply that
\begin{align}\label{eq:h_n stri}
\|e^{iW_t\Delta}\F^{-1}h_{n}\|_{L^{4}( [0,T_{0}]\times\T^{2})}
&=(\int_{0}^{T_{0}}\int_{\T^{2}}\left|e^{iW_t\Delta}\F^{-1}h_{n}\right|^{4}\, \dd x \dd t)^{1/4}\\
&\lesssim
(T_0\log(N) + N^{2\epsilon}\sqrt{T_0})^{1/4}\|\lambda_{j}\|_{\ell_{j\leq m}^{2}} \nonumber\\
&\lesssim
(T_0\log(N) + N^{2\epsilon}\sqrt{T_0})^{1/4}\|f_{n}\|_{\ell^{2}(\Z^{2})}.\nonumber
\end{align}
Writing $f=\sum_{n=0}^{\infty}(f_{n}-f_{n+1})=\sum_{n=0}^{\infty}h_{n}$,
by \eqref{eq:h_n stri} and \eqref{eq:f_n _2 < 2^-n f_2}, we have
\begin{align*}
\|e^{iW_t\Delta}\F^{-1}f\|_{L_{t,x}^{4}\left([0,T_{0}]\times\T^{2}\right)} & \leq\sum_{n=0}^{\infty}\|e^{iW_t\Delta}\F^{-1}h_{n}\|_{L_{t,x}^{4}([0,T_{0}]\times\T^{2})}\\
 & \lesssim 
 (T_0\log(N) + N^{2\epsilon}\sqrt{T_0})^{1/4}\sum_{n=0}^{\infty}\|f_{n}\|_{\ell^{2}(\Z^{2})}\\
 &\lesssim
 (T_0\log(N) + N^{2\epsilon}\sqrt{T_0})^{1/4} \sum_{n=0}^{\infty}2^{-n} \|f\|_{\ell^{2}(\Z^{2})}\\
 &\lesssim 
 (T_0\log(N) + N^{2\epsilon}\sqrt{T_0})^{1/4} \|f\|_{\ell^{2}(\Z^{2})}.
\end{align*}
This is bounded by $(2\sqrt{T_0}N^{2\epsilon})^{1/4}\|f\|_{\ell^2(\Z^2)} \lesssim (\sqrt{T_0}N^{2\epsilon})^{1/4}\|f\|_{\ell^2(\Z^2)}$.

This completes the proof of pathwise Strichartz estimates. Now we address the stochastic Strichartz estimates. Assume that $W_t = B^H_t$ is a fractional Brownian motion on a probability space $(\Omega, \FF, P)$ and $f:\Omega \times \Z^2 \rightarrow \C$ is a given function. We can repeat everything appearing in the proof of pathwise Strichartz estimates, i.e. the reduction to $f$ nonnegative, the sequences of functions $\{f_n\}, \{g_n\}, \{h_n\}$, the sets $S_j$, $E_j$, etc. The only difference is that now all of these objects are dependent on the random parameter $\omega \in \Omega$. We repeat the estimate of $h_n$ from \eqref{Intermediary28} except now we must estimate in expectation.
\begin{align*}
\mathbb{E}\Bigg{[}\int_{0}^{T_{0}}\int_{\T^{2}}\bigg{|}e^{iW_t\Delta}\F^{-1}h_{n}\bigg{|}^{4} \dd x \dd t \Bigg{]} 
 & = \E\Bigg{[}\int_{0}^{T_{0}}\F(|e^{iB^H_t\Delta}\F^{-1}h_{n}|^{4})(0)\dd t \Bigg{]} \\
  & =
  \E\Bigg{[}\sum_{Q\in \Q}h_{n}(Q)\cdot\int_{0}^{T_{0}} e^{-iW_t\tau_{Q}}\dd t \Bigg{]}  \\
  & =
  \E\Bigg{[}\sum_{Q \in \mathcal{Q}^0}h_n(Q)T_0 + \sum_{\tau \neq 0}\sum_{Q \in \mathcal{Q}^{\tau}}h_n(Q) \int_0^{T_0} e^{-iB^H_t \tau} \dd t\Bigg{]}.  
\end{align*}
Because $h_n$ is a measurable function of $f$, like $f$, $h_n$ will be independent of the increments of $B^H_t$. Using the fact that $\E[e^{iB^H_t \tau}] = e^{-\frac{\tau^2}{2}t^{2H}}$ 
\begin{align}
&  \E\Bigg{[}\sum_{Q \in \mathcal{Q}^0}h_n(Q)T_0 + \sum_{\tau \neq  0}\sum_{Q \in \mathcal{Q}^{\tau}}h_n(Q) \int_0^{T_0} e^{-iB^H_t \tau} \dd t \Bigg{]}\notag\\
&=
 \E\Bigg{[}\sum_{Q \in \mathcal{Q}^0}h_n(Q)T_0 \Bigg{]}+ \E\Bigg{[}\sum_{\tau \neq 0}\sum_{Q \in \mathcal{Q}^{\tau}}h_n(Q)\Bigg{]} \int_0^{T_0} \E \big{[}e^{-iB^H_t \tau}\big{]} \dd t \notag\\
 &=
\E\Bigg{[}\sum_{Q \in Q^0} h_n(Q)T_0 \Bigg{]}+ \E\Bigg{[}\sum_{\ell=1}^{\log(N^2)} \sum_{\tau = 2^{\ell-1}}^{2^{\ell}} \sum_{Q \in Q^{\tau}}h_n(Q)\int_0^{T_0}e^{-\frac{\tau^2}{2}t^{2H}} \dd t \Bigg{]}. \label{Intermediary25}
\end{align}
We bound the integral in this expression by approximating $e^{ -\frac{\tau^2}{2} t^{2H}}$ as constant on the intervals in time where $j-1 \leq \frac{\tau^2}{2} t^{2H} \leq j$. Using the mean value theorem we obtain that
\begin{align*}
 \int_0^{T_0}e^{-\frac{\tau^2}{2}t^{2H}} \dd t  
&\lesssim 
\sum_{j=1}^{\frac{\tau^2}{2} T_0^{2H}} e^{-j-1}\bigg{[}\big{(}\frac{2j}{\tau^2}\big{)}^{\frac{1}{2H}} - \big{(}\frac{2(j-1)}{\tau^2}\big{)}^{\frac{1}{2H}}  \bigg{]} \\
&\lesssim
\sum_{j=1}^{\frac{\tau^2}{2} T_0^{2H}} e^{-j} \frac{1}{\tau^2 H}\big{(} \frac{2j}{\tau^2}\big{)}^{\frac{1}{2H}-1} 
\notag \\
&\lesssim
\frac{1}{\tau^{1/H}}\sum_{j=1}^{\frac{\tau^2}{2}T_0^{2H}}\frac{e^{-j}(2j)^{\frac{1}{2H}-1}}{H} 
\\
&\lesssim_H
\frac{1}{\tau^{1/H}}.
\end{align*}
Combining this with the trivial bound $\int_0^{T_0} e^{-\frac{\tau^2}{2}t^{2H}} \dd t \leq T_0$ we obtain that
\begin{equation*}
\int_0^{T_0} e^{-\frac{\tau^2}{2}t^{2H}} \dd t \lesssim_{H} \min(\frac{1}{\tau^{1/H}}, T_0).
\end{equation*}
Plugging this into \eqref{Intermediary25}, we bound \eqref{Intermediary25} by
\begin{align*}
&\E\Bigg{[}\sum_{Q \in Q^0} h_n(Q)T_0\Bigg{]}  + \E\Bigg{[}\sum_{\ell=1}^{\log(N^2)} \sum_{\tau = 2^{\ell-1}}^{2^{\ell}} \sum_{Q \in Q^{\tau}}h_n(Q)\min(\frac{1}{\tau^{1/H}}, T_0) \Bigg{]}
\notag 
\end{align*}
Using Proposition \ref{Prop:PropositionForStrichartz} we bound this by 
\begin{equation}
\E[\log(N)T_0\|\lambda_j\|_{\ell^2}^4] + \E[\|\lambda_j\|_{\ell^2}^4]\sum_{\ell=1}^{\log(N^2)} \min(\frac{1}{2^{\ell(\frac{1}{H}-1)}}, 2^{\ell}T_0) \label{Intermediary27}
\end{equation}
If $N^2 \geq T_0^{-H} $ then 
\begin{align*}
&\sum_{\ell=1}^{\log(N^2)} \min(\frac{1}{2^{\ell(\frac{1}{H}-1)}}, 2^{\ell}T_0) 
\\
&\lesssim
\sum_{\ell=1}^{\log(T_0^{-H})} 2^{\ell}T_0 + \sum_{\ell = \log(T_0^{-H})}^{\log(N^2)} \frac{1}{2^{\ell(\frac{1}{H}-1)}}
\\
&\lesssim
T_0^{1-H} + \frac{N^{2(1-\frac{1}{H})-T_0^{1-H}}}{2^{\frac{H-1}{H}}-1} \\
&\lesssim
T_0^{1-H}.
\end{align*} 
If $N^2 \leq T_0^{-H}$ then the term where $\frac{1}{2^{{\ell}(\frac{1}{H}-1)}} < 2^{\ell}T_0 $ vanishes. Substituting this into \eqref{Intermediary27} we bound \eqref{Intermediary27} by
\begin{equation*}
(\log(N)T_0 + T_0^{1-H})\E[\|\lambda_j\|_{\ell^2}^4].
\end{equation*}
The rest of the analysis is the same as that appearing in the proof of pathwise Strichartz estimates starting from \eqref{Intermediary29}.

\end{proof}

Recall from Theorem \ref{Thm:BrownianIsRhoIrregular} that a standard Brownian motion, and more generally a fractional Brownian motion with Hurst index $H \leq 1/2$, satisfies the hypotheses of Theorem \ref{Thm:PathwiseStrichartz}. We immediately obtain the following corollary.

\begin{cor}
If $B^H_t$ is a fractional Brownian motion with Hurst index $H \leq 1/2$ then for all rectangles $S = \{\xi: |\xi_1-a|,|\xi_2-b| \leq N\}$,  almost surely the paths $t \mapsto B^H_t(\omega)$ satisfy the strichartz estimate
\begin{equation*}
\|e^{iB^H_t(\omega) \Delta}P_Sf\|_{L^4(I \times \T^2)} \lesssim_{\epsilon, \omega}(N^{2\epsilon}\sqrt{|I|})^{1/4} \|f\|_{L^2(\T^2)}.
\end{equation*}
The implicit constant in this inequality will depend on the choice of path $t \mapsto B^H_t(\omega)$ for the fractional Brownian motion.
\end{cor}

\section{Function Spaces}\label{Sec:FunctionSpaces}

In this section we discuss the function spaces that we use to apply a Picard iteration and obtain a solution to $\eqref{Eq:ModulatedNLSDuhamel}$. As far as the author is aware these spaces were first introduced by Wiener in \cite{ctx23497407350006761NWienerVarDefn} and first used in the context of dispersive PDE in \cite{KochTataru2005IntroductionOfVariaSpaces}. Our notation is based on \cite{IncidenceGeometryStrichartz, UpVpEstimates, XsYsEstimates}. Let $X$ be any Banach space. In this paper $X$ will usually be the Sobolev space $H^s(\T^2)$. The $p$-variation spaces that we will discuss are built on the notion of a partition. In the following let $a,b \in [-\infty, \infty]$ and $a < b$. In this paper we will mostly consider the case $a=0$ and $b=T>0$.

\begin{df}
 Define $\mathcal{Z}$ to be the set of all finite sequences of increasing times $a = t_0 < \cdots < t_k = b $. Elements of $\mathcal{Z}$ are called partitions.
\end{df}  
\begin{df}\label{Updefn}
Let $1 < p < \infty$. We say that $v \in L^{\infty}(\R, X)$ is a $U^p(X)$ atom if there exists a partition $\pi = \{t_0 < \cdots < t_k\} \in \mathcal{Z}$ and $\phi_j \in X$ satisfying $(\sum_{j=1}^k \|\phi_j\|_X^p)^{1/p}=1$ such that
\begin{equation*}
v = \sum_{j=1}^k 1_{[t_{j-1}, t_j)}\phi_j.
\end{equation*}
We define $U^p(X)$ to be the set of all $v \in L^{\infty}(\R, X)$ such that
\begin{align*}
&\|v\|_{U^p(X)}
:=
\inf\left\{\sum_{j=1}^{\infty}|\lambda_j|: v = \sum_{j=1}^{\infty} \lambda_jv_j \text{ and for all j, }v_j \text{ is a }U^p(X) \text{ atom} \right\} < \infty,
\end{align*} 
\end{df}
\begin{df}\label{Vpdefn}
The space $V^p(X)$ consists of the set of all functions $v: [a,b] \rightarrow X$ such that the norm 
\begin{equation*}
\|v\|_{V^p(X)}:= \sup_{\{t_j\}_{j=0}^M \in \mathcal{Z}} \left(\sum_{j=0}^{M} \|v(t_j)-v(t_{j-1})\|_{X}^p\right)^{1/p}
\end{equation*}
is finite. Note that we use the convention $v(t_{-1})=0$. This convention is necessary because otherwise all nonzero constant functions $v(t)=v \in X$ would have norm $0$ in the space $V^p(X)$. One should be careful since many authors make slightly different conventions in order to address this. See \cite[Section 2]{XsYsEstimates}, \cite[Section 2]{UpVpEstimates}, and \cite[Section 4.1]{UpVptextbook}.
\end{df}
\begin{rem}
Both the $U^p$ and $V^p$ spaces implicitly depend on the interval that they are defined on. Throughout the paper, when the Banach space $X$ is obvious, we may use the notation $U^p([0,T])$ or $V^p([0,T])$ to emphasize this interval. In other cases where neither the Banach space or the interval needs to be emphasized we may just write $V^p$ or $U^p$.
\end{rem}
\begin{rem}
For each $1 \leq p < \infty $ the spaces $U^p(X)$ and $V^p(X)$ are Banach spaces. If $p > 1$  the dual of $U^p(X)$ is $V^{p'}(X^*) $ where $\frac{1}{p}+\frac{1}{p'}=1$ and $X^*$ is the dual of $X$. For a discussion of the general theory of these spaces see \cite{UpVptextbook,UpVpEstimates,XsYsEstimates}.
\end{rem}
The contraction argument that is used to construct solutions to \eqref{Eq:ModulatedNLSDuhamel} will utilize the following types of spaces:
\begin{df}\label{Defn:YSpace} Let $T > 0$, $f(t, x): [0,T] \times \T^2 \rightarrow \mathbb{C}$ a function and $B$ a Banach space of functions on the interval $[0,T]$. If $W_t:[0,T]\rightarrow \R$ is a function we can define the space $Y^s(B)$, depending on $W_t$, as the space of functions such that
\begin{equation*}
\|f\|_{Y^s(B)} := \left(\sum_{k \in \Z^2} \langle k \rangle^{2s}\|e^{iW_t|k|^2}f_k(t)\|_B^2\right)^{1/2} < \infty,
\end{equation*}
In this paper we will consider the cases where $B= V^p([0,T])$ or $B=U^p([0,T])$ for $p > 1$.
\end{df}

In Section \ref{Sec:Trilinear} we discuss Strichartz type estimates that are adapted to these spaces. In an intermediary step of the proof of Theorem \ref{Thm:LocalWellPosedness} we consider the following norms
\begin{df}\label{Defn:U^p_WSpace}
\begin{equation*}
\|u\|_{U^p_{W}(H^s_x)} := \|e^{-iW_t\Delta}u\|_{U^p(H^s_x)},
\end{equation*}
\end{df}
and
\begin{df}\label{Defn:V^p_WSpace}
\begin{equation*}
\|u\|_{V^p_{W}(H^s_x)} := \|e^{-iW_t\Delta}u\|_{V^p(H^s_x)}.
\end{equation*}    
\end{df}
The following proposition relates these norms to the $Y^s(U^p)$ and $Y^s(V^p)$ spaces.

\begin{prop}\label{Prop:FunctionEmbeddings}
If $p \geq 2$ and $W_t$ is a given function then the following are continuous embeddings
\begin{equation*}
 Y^s(U^p) \hookrightarrow Y^s(V^p) \hookrightarrow V^p_{W}(H^s_x). 
\end{equation*}
\end{prop}

\begin{proof}
First we show that $\|u\|_{V^p_{W}(H^s_x)} \lesssim \|u\|_{Y^s(V^p)}$. Using that $p \geq 2$ we apply Minkowski's inequality:
\begin{align*}
& \|u\|_{V^p_{W}(H^s_x)} \\
&=
\sup_{\pi \in \mathcal{Z}} \bigg{(} \sum_{t_j \in \pi} \|e^{-iW_{t_j}\Delta }u(t_j) - e^{-iW_{t_{j-1}}\Delta }u(t_{j-1})\|_{H^s_x}^p \bigg{)}^{\frac{1}{p}} \\
&=
\sup_{\pi} \Bigg{[}\sum_{t_j \in \pi} \bigg{(} \sum_{k \in \Z^2}\langle k \rangle^{2s}|e^{iW_{t_j}|k|^2} u_k(t_j) - e^{iW_{t_{j-1}}|k|^2} u_k(t_{j-1})|^2 \bigg{)}^{\frac{p}{2}} \Bigg{]}^{\frac{2}{p}\cdot\frac{1}{2}} \\
&\leq
\sup_{\pi} \bigg{[} \sum_{k \in \Z^2}\langle k \rangle^{2s}\left(\sum_{t_j \in \pi}|e^{iW_{t_j}|k|^2} u_k(t_j) - e^{iW_{t_{j-1}}|k|^2} u_k(t_{j-1})|^p\right)^{\frac{2}{p}} \bigg{]}^{\frac{1}{2}} \\
&\leq
\|u\|_{Y^s(V^p)}.
\end{align*}

The  embedding $Y^s(U^p) \hookrightarrow Y^s(V^p)$ immediately follows from the fact that $U^p(X) \hookrightarrow V^p(X)$ is a continuous embedding for any Banach space $X$. This can be shown by checking that if $u$ is a $U^p(X)$ atom then $\|u\|_{V^p} \lesssim 1$. For further reference see \cite{UpVptextbook}.

\end{proof}

The following result lists some basic properties of the $Y^s$ spaces that will be necessary for the contraction argument we perform in Section \ref{Sec:Local}. The proofs are similar to those appearing in \cite[Section 2]{XsYsEstimates}. Because our $Y^s$ spaces are based on the Fourier multiplier $e^{iW_t|k|^2}$ as opposed to $e^{it|k|^2}$ there are some slight changes that need to be made so we reprove the results here.
\begin{prop} The $Y^s$ norms have the following properties.

If $A$ and $B$ are disjoint subsets of $\Z^2$ then  
\begin{equation}\label{YsSquaring Estimate} 
\|P_Af\|_{Y^s(V^p)}^2 + \|P_Bf\|_{Y^s(V^p)}^2 = \|P_{A \sqcup B}f\|_{Y^s(V^p)}^2.
\end{equation}

If $f \in Y^s(V^p)$ where $p \geq 2$ and $\mathcal{I}(f) = \int_0^t e^{i(W_t-W_{t'})\Delta}f(t') \dd t'$ then 
\begin{equation}\label{YsInhomogeneousEstimate} 
\| 1_{[0,T]} \mathcal{I}(f) \|_{Y^s(V^p)} \lesssim \bigg{|} \sup_{\|v\|_{Y^{-s}(V^{p'})}=1} \int_0^T \int_{\T^2}  f\overline{v} \dd x \dd t \bigg{|}.
\end{equation}

If $f \in H^s$ then
\begin{equation}\label{YsHsEstimate}\|1_{[0,T]} e^{i W_t \Delta} f \|_{Y^s(V^p)} \approx \|f\|_{H^s}.
\end{equation}

and for functions $f \in Y^s(V^p)$
\begin{equation}\label{YsLInfinityLowerBoundEstimate}
\|1_{[0,T]}f\|_{Y^s(V^p)} \gtrsim \|f\|_{L^{\infty}([0,T],\, H^s)}.
\end{equation}
\end{prop}

\begin{proof}
The estimate \eqref{YsSquaring Estimate} follows immediately from expanding definitions. The estimate \eqref{YsHsEstimate} follows from the fact that the $V^p$ norm of a constant is that constant, and \eqref{YsLInfinityLowerBoundEstimate} follows from the fact that $\|g\|_{L^{\infty}([0,T],X)} \leq \|g\|_{V^p(X)}$ holds for any Banach space $X$ and function $g:[0,T]\rightarrow X$. Equation \eqref{YsInhomogeneousEstimate} is more difficult so we prove it in detail. Let $f \in Y^s(V^p)$. Using Proposition \eqref{Prop:FunctionEmbeddings} and duality,
\begin{align}
&\|\mathcal{I}(f)\|_{Y^s(V^p)}
\leq 
\|\mathcal{I}(f)\|_{Y^s(U^p)} \notag\\
&=
\bigg{(} \sum_{k \in \Z^2}\langle k \rangle^{2s}\|\int_0^t 1_{[0,T]}e^{iW_{t'}|k|^2}f_k(t') \|_{U^p}^2 \bigg{)}^{1/2} \notag\\
&=
\sup_{\|\{a_k\}\|_{\ell^2(\Z^2)=1}}\bigg{|}  \sum_{k \in \Z^2}a_k\langle k \rangle^{s}\|\int_0^t 1_{[0,T]}e^{iW_{t'}|k|^2}f_{k}(t') \|_{U^p}   \bigg{|}. \label{Intermediary14}
\end{align}
Because $V^{p'}$ is dual to $U^p$ we use \cite[Theorem 2.8 and Proposition 2.10] {XsYsEstimates} to deduce that $\eqref{Intermediary14}$ is bounded above by 
\begin{align}\label{Intermediary15}
\sup_{\|\{a_k\}\|_{\ell^2(\Z^2)=1}} \bigg{|}  \sum_{k \in \Z^2} a_k\langle k \rangle^{s} \sup_{\|v\|_{V^{p'}}=1} \int_0^T e^{iW_{t}|k|^2} f_k(t)v(t) \dd t \bigg{|}.
\end{align}
Let $\epsilon > 0$. For each $k$ choose $a_k \in \C$ and $v_k \in V^{p'}$ such that $\|\{a_k\}\|_{\ell^2(\Z^2)} = 1$, $\|v_k\|_{V^{p'}} = 1$, and 
\begin{align*}
&\eqref{Intermediary15} \leq \epsilon + \bigg{|}\sum_{k \in \Z^2} a_k\langle k \rangle^{s} \int_0^T e^{iW_{t}|k|^2} f_k(t)v_k(t) \dd t  \bigg{|} \\
&=
\epsilon + \bigg{|} \int_{0}^T \int_{\T^2}f\overline{v} \dd x \dd t \bigg{|}.
\end{align*}
where $v(t, x) = \sum_{k \in \Z^2}\overline{a_k}\langle k \rangle^se^{-iW_t|k|^2}\overline{v_k(t)}e^{ikx}$. The fact that $\|\{a_k\}\|_{\ell^2(\Z^2)} = \|v_k\|_{V^{p'}} = 1$ for all $k$ implies that $\|v(t, x)\|_{Y^{-s}(V^{p'})}=1$.  Since $\epsilon $ was arbitrary, \eqref{YsInhomogeneousEstimate} follows.
\end{proof}
\section{A trilinear Estimate for \texorpdfstring{$\mathcal{I}$}{Lg}}\label{Sec:Trilinear}
In the beginning of this section we show that the pathwise Strichartz estimates proven in Theorem \ref{Thm:PathwiseStrichartz} can be adapted to the space $Y^0(V^2)$. We use this to prove a trilinear estimate for the operator $\mathcal{I}$, defined in (\ref{YsInhomogeneousEstimate}). This result enables us to prove the existence of local solutions to \eqref{Eq:ModulatedNLSDuhamel} in Section \ref{Sec:Local}.

For $N \in 2^{\N}$, let $\mathcal{C}_N$ denote the set of all squares with side length $N$.
\begin{equation*}
    \CC_N := \{ (0, N]^2 +N\xi_0 \cap \Z^2: \xi_0 \in \Z^2\}.
\end{equation*}

\begin{lem}\label{Lem:YsStrichartz}
Let $N \in 2^{\N}$, $\epsilon > 0$, and $I$ be an interval such that $|I| = T_0 $. For all squares $C \in \CC_N$ and $u \in Y^0(V^2)$ we have 
\begin{equation*}
    \|P_Cu\|_{L^4(I \times \T^2)}\lesssim_{W_t, \epsilon} (N^{2\epsilon}\sqrt{T_0})^{1/4}\|u\|_{Y^0(V^2)}.
\end{equation*}
\end{lem}

\begin{proof}
The proof is similar to \cite[Lemma 4.3]{IncidenceGeometryStrichartz}, but with some adjustments to account for the modulation and keep track of the exact dependency of implicit constants on $N$ and $T_0$. Let $u = \sum_{t_j \in \pi}1_{[t_{j-1}, t_j)}e^{iW_t\Delta}\phi_j$ be a $U^4_{W}(H^0_x)$ atom. Using Theorem \ref{Thm:PathwiseStrichartz}, 
\begin{align*}
\|P_Cu\|_{L^4_{I, x}}^4 
= 
\sum_{t_j \in \pi} \|[1_{t_{j-1}, t_j)}e^{iW_t\Delta}P_C\phi_j\|_{L^4_{I, x}}^4
\lesssim 
(N^{2\epsilon}\sqrt{T_0})\sum_{t_j \in \pi}\|\phi_j\|_{L^2_x}^4
=
(N^{2\epsilon}\sqrt{T_0}).
\end{align*}
If $u = \sum_{j=1}^{\infty}\lambda_j u_j$ is an atomic decomposition of $u$ such that $\sum_{j=1}^{\infty}\lambda_j \approx \|u\|_{U_W^4(H^0_x)}$ then 
\begin{align*}
\|P_Cu\|_{L^4_{t, x}} 
&\lesssim
\sum_{j=1}^{\infty} \lambda_j \|P_Cu_j\|_{L^4(I \times \T^2)}\\
&\lesssim
(N^{2\epsilon}\sqrt{T_0})^{1/4}\sum_{j=1}^{\infty} \lambda_j \notag\\
&\lesssim
(N^{2\epsilon}\sqrt{T_0})^{1/4}\|u\|_{U^4_{W}(H^0_x)}  \\
&\lesssim (N^{2\epsilon}\sqrt{T_0})^{1/4}\|u\|_{V^2_W(H^0_x)} \notag\\
&\lesssim 
(N^{2\epsilon}\sqrt{T_0})^{1/4}\|u\|_{Y^0(V^2)},
\end{align*}
where the last line follows from Proposition \ref{Prop:FunctionEmbeddings}.
\end{proof}

In order to construct solutions to \eqref{Eq:ModulatedNLSDuhamel} we define the following space.
\begin{df}\label{Defn:Zn}
For $N \in 2^{\N}$ and $\epsilon, \beta > 0$ we set $I_N := [0, \beta/N^{4\epsilon}]$. Let $Z_N$ be the norm, which depends on s,
\begin{equation*}
\|u\|_{Z_N} := \|1_{I_N} \cdot u\|_{Y^0(V^2)} + N^{-s}\|1_{I_N} \cdot u\|_{Y^s(V^2)}.
\end{equation*}
\end{df}
The purpose of Definition \ref{Defn:Zn} is to define a space where the initial data for \eqref{Eq:ModulatedNLSDuhamel} is small and $\mathcal{I}$, defined in \eqref{YsInhomogeneousEstimate}, is a contraction. The primary role of $N$ is to control the high frequency terms of the initial data while the primary role of $\beta$ is to make $\mathcal{I}(\cdot)(t) := \int_0^t e^{i(W_t-W_t') \Delta}(\cdot) \dd t'$ more contractive if the initial data is large in $L^2(\T^2)$. The utility of the parameter $\beta$ is shown in the following lemma. 
\begin{lem}\label{Lem:TrilinearEstimate} For $0 < s \leq 1$ and an integer $N >> \frac{1}{s}$ 
\begin{equation*}
\|\mathcal{I}(u_1u_2u_3)\|_{Z_N} \lesssim \sqrt{\beta} \|u_1\|_{Z_N}\|u_2\|_{Z_N}\|u_3\|_{Z_N},\label{Intermediary5}
\end{equation*}
where each $u_j$ could be replaced by its complex conjugate. Note that the implicit constant depends on $s$ due to the fact that we use Theorem \ref{Thm:PathwiseStrichartz} with $\epsilon < s/5$. 
\end{lem}

\begin{proof}
This proof is heavily adapted from \cite[Lemma 4.4]{IncidenceGeometryStrichartz}. Let $k_s \approx 1/s$ be an integer and $\epsilon = s/5$. In this proof we use $2^{k_s}$-adic cutoffs. For $M \in 2^{k_s\N}$ we denote
\[
P_{\sim M}u = u_{\sim M} = u_{< 2^{k_s}M} - u_{<M}.
\]
Similarly, sums of the form $\sum_{M}$ denote sums over $M \in 2^{k_s\N}$.  Since $\|1_{I_N} \cdot u\|_{Z_{\tilde{N}}} \approx \|u\|_{Z_{N}}$ holds for $\tilde{N} \in [2^{-k_s}N,N]$, we assume further that $N \in 2^{k_s\N}$. By \eqref{YsInhomogeneousEstimate} and Definition \ref{Defn:Zn}, Lemma \ref{Intermediary5} is reduced to showing
\begin{equation}\label{Intermediaryeq1}
  \bigg{|} \int_{I_N \times \T^2} u_1u_2u_3 \cdot v_{< N} \dd x \dd t  \bigg{|} \lesssim \sqrt{\beta}\|u_1\|_{Z_N}\|u_2\|_{Z_N}\|u_3\|_{Z_N}\|v\|_{Y^0(V^2)}, 
\end{equation}
and
\begin{equation}\label{Intermediaryeq2}
  \bigg{|} \int_{I_N \times \T^2} u_1u_2u_3 \cdot v_{\geq N} \dd x \dd t  \bigg{|} \lesssim \sqrt{\beta}\|u_1\|_{Z_N}\|u_2\|_{Z_N}\|u_3\|_{Z_N}N^s\|v\|_{Y^{-s}(V^2)}.
\end{equation}
For $M \geq N$, we partition the interval $I_N:=[0, \beta/N^{4\epsilon}]$ into intervals of length approximately $\beta/M^{4\epsilon}$. Applying Lemma \ref{Lem:YsStrichartz} to each of these intervals we obtain, for each square $C \in \CC_M$,
\begin{equation}
\|1_{I_N}P_Cu\|_{L^4_{t, x}} 
\lesssim 
\bigg{(}M^{2\epsilon}\frac{\sqrt{\beta}}{M^{2\epsilon}}\bigg{)}^{1/4}\bigg{(} \frac{M^{4\epsilon}}{N^{4\epsilon}}\bigg{)}^{1/4}\|u\|_{Y^0(V^2)}
\lesssim
\beta^{1/8}\bigg{(}\frac{M^{\epsilon}}{N^{\epsilon}}\bigg{)}\|u\|_{Y^0(V^2)}.
\label{Intermediary6}
\end{equation}
Applying a $2^{k_s}$-adic decomposition to $u$ and using this estimate on each part
\begin{align*} 
   \|1_{I_N} u \|_{L^4_{t, x}}  
&\lesssim \beta^{1/8}\|1_{I_N}u\|_{Y^0(V^2)} + \beta^{1/8}\sum_{M \geq N} \bigg{(}\frac{M^{\epsilon}}{N^{\epsilon}}\bigg{)}\|1_{I_N}u_{\sim M}\|_{Y^0(V^2)} \\
&=
\beta^{1/8}\|1_{I_N}u\|_{Y^0(V^2)} + \beta^{1/8}N^{-s}\sum_{M \geq N} \bigg{(}\frac{M^{\epsilon}}{N^{\epsilon}}\bigg{)}\frac{M^{-s}}{N^{-s}}\|1_{I_N}u_{\sim M}\|_{Y^s(V^2)} \\
&=
\beta^{1/8}\|1_{I_N}u\|_{Y^0(V^2)} + \beta^{1/8}N^{-s}\|1_{I_N}u\|_{Y^s(V^2)} \\
&=
\beta^{1/8}\|u\|_{Z_N},
\end{align*}
where we use that $\epsilon < s$ to deduce convergence of the sum. Applying this and H\"{o}lder's inequality to the left hand side of \eqref{Intermediaryeq1} proves \eqref{Intermediaryeq1}. To estimate the left hand side of \eqref{Intermediaryeq2} we note that
\begin{align*}
\bigg{|} \int_{I_N \times \T^2} u_1u_2u_3 \cdot v_{\geq N} \dd x \dd t \bigg{|} 
&\lesssim
  \sum_{K \geq N} \bigg{|} \int_{I_N \times \T^2} u_{1_{\geq K/4}}u_2u_3 \cdot v_{\sim K} \dd x \dd t  \bigg{|} \\
&+
  \sum_{K \geq N} \bigg{|} \int_{I_N \times \T^2} u_{1_{\leq K/4}}u_{2_{\geq K/4}}u_3 \cdot v_{\sim K} \dd x \dd t  \bigg{|} \\
&+
  \sum_{K \geq N} \bigg{|} \int_{I_N \times \T^2} u_{1_{\leq K/4}}u_{2_{\leq K/4}}u_{3_{\geq K/4}} \cdot v_{\sim K} \dd x \dd t  \bigg{|},
\end{align*}
and therefore it suffices to estimate expressions of the form
\begin{align}
\sum_{K \geq N} \bigg{|} \int_{I_N \times \T^2} u_1u_2w_{\geq K/4} \cdot v_{\sim K} \dd x \dd t  \bigg{|}  
&\lesssim 
\sum_{M \geq 0}\sum_{K \geq N}\sum_{K/4\leq L \leq K+M} \|P_{\sim M}(u_1 u_2)\|_{L^2(I_N \times \T^2)} \notag\\
&\hspace{24ex}\|P_{\sim M}(w_{\sim L} \cdot v_{\sim K})\|_{L^2(I_N \times \T^2)}.  \label{intermediary7}
\end{align}
To estimate the terms in \eqref{intermediary7} we tile $\Z^2$ using cubes with side length $M$. For generic $f,g \in L^2(\T^2)$,
\begin{align*}
\|1_{I_N} P_{\sim M}(fg)\|_{L^2_{t, x}} 
&\leq
\sum_{\substack{C_1, C_2 \\ \text{dist}(C_1,C_2) \leq M}} \|1_{I_N} P_{C_1}f P_{C_2}g \|_{L^2_{t, x}} \\
&\leq
\sum_{\substack{C_1, C_2 \\ \text{dist}(C_1,C_2) \leq M}} \|1_{I_N} P_{C_1}f\|_{L^4_{t, x}} \|1_{I_N} P_{C_2}g \|_{L^4_{t, x}}\\
\end{align*}
Given a cube $C$ in a tiling of $\Z^2$ by cubes of side length $M$, there are only finitely many other cubes in the tiling that are within a distance $M$ of $C$. Moreover, this bound is independent of $M$. Using this, \eqref{Intermediary6}, and $\eqref{YsSquaring Estimate}$,
\begin{align*}
&\sum_{\substack{C_1, C_2 \\ \text{dist}(C_1,C_2) \leq M}} \|1_{I_N} P_{C_1}f\|_{L^4_{t, x}} \|1_{I_N} P_{C_2}g \|_{L^4_{t, x}} \\
&\lesssim 
(\beta)^{1/4}(1+\frac{M^{2\epsilon}}{N^{2\epsilon}}) \bigg{(} \sum_{C_1} \|P_{C_1}f\|_{Y^0(V^2)}^2 \sum_{C_2} \|P_{C_2}f\|_{Y^0(V^2)}^2 \bigg{)}^{1/2} \\
&=
(\beta)^{1/4}(1+\frac{M^{2\epsilon}}{N^{2\epsilon}})\|f\|_{Y^0(V^2)}\|g\|_{Y^0(V^2)},
\end{align*}
We can use this to bound \eqref{intermediary7} by
\begin{align}
&\sqrt{\beta}\sum_{M \geq 0}\sum_{K \geq N}\sum_{K/4\leq L \leq K+M} \big{(}\|P_{\geq M/4}u_1\|_{Y^0(V^2)}\|u_2\|_{Y^0(V^2)}\notag\\
&\hspace{4ex}+
\|u_1\|_{Y^0(V^2)}\|P_{\geq M/4}u_2\|_{Y^0(V^2)} \big{)}(1+\frac{M^{4\epsilon}}{N^{4\epsilon}})\|w_{\sim L}\|_{Y^0(V^2)}\|v_{\sim K}\|_{Y^0(V^2)} \notag\\
& \lesssim
\sqrt{\beta}\sum_{M \geq 0}\sum_{K \geq N}\sum_{K/4\leq L \leq K+M}(L/K)^{-s}(1+\frac{M^{4\epsilon}}{N^{4\epsilon}}) (1+\frac{M}{N})^{-s} \notag\\
&
\hspace{20ex}\|u_1\|_{Z_N} \|u_2\|_{Z_N}\|w_{\sim L}\|_{Y^s(V^2)}\|v_{\sim K}\|_{Y^{-s}(V^2)} \notag\\
&\leq
\sqrt{\beta}\|u_1\|_{Z_N} \|u_2\|_{Z_N}\sum_{K \geq N}\|v_{\sim K}\|_{Y^{-s}(V^2)}\sum_{K/4\leq L }(L/K)^{-s}\|w_{\sim L}\|_{Y^s(V^2)}.\notag
\end{align}
Here we use that $\epsilon = s/5 < s/4$ to guarantee convergence of the sum in $M$. If $K = 2^{kk_s}$, $L = 2^{\ell k_s}$, and $N=2^{nk_s}$ then 
\begin{align*}
&\sum_{K \geq N}\|v_{\sim K}\|_{Y^{-s}(V^2)}\sum_{K/4\leq L }(L/K)^{-s}\|w_{\sim L}\|_{Y^s(V^2)} \notag \\
&\approx \sum_{k \geq n}\|v_{\sim 2^{kk_s}}\|_{Y^{-s}(V^2)}\sum_{k \leq \ell }2^{k-\ell}\|w_{\sim 2^{\ell k_s}}\|_{Y^s(V^2)}.
\end{align*}
Applying Young's inequality and $\eqref{YsSquaring Estimate}$ proves the desired result.
\end{proof}

\section{Local Well-Posedness}\label{Sec:Local}
We construct solutions to \eqref{Eq:ModulatedNLSDuhamel} by performing a contraction in a ball in the space $Z_N$ intersected with functions that are continuous in time into $H^s$. Note that this metric space will be complete due to the fact that the set of continuous $V^p(X)$ functions are always closed in $V^p(X)$ for any Banach space $X$. We use our Strichartz estimates from Theorem \ref{Thm:PathwiseStrichartz} to control the linear part of \eqref{Eq:ModulatedNLSDuhamel} and we use Lemma \ref{Lem:TrilinearEstimate} to control the nonlinear part of \eqref{Eq:ModulatedNLSDuhamel}.

\begin{thm}\label{Thm:LocalWellPosedness} Suppose that for all $\rho < 1$, $W_t$ is a $(\rho, 1/2)$-irregular function, as in Definition \ref{Defn:RhoGammaIrregular}. Let $s > 0$, $N >> \frac{1}{s}$, $u_0 \in H^s(\T^2)$ be fixed, and $\Lambda \geq \|u_0\|_{H^0} + N^{-s}\|u_0\|_{H^s}$. Recall from Definition \ref{Defn:Zn} that for any given $N$, $I_N := \beta/N^{4\epsilon}$ where $\beta$ is to be chosen depending on $\Lambda$. Define
\begin{equation*}
X_{N}:= \{u \in C^0([0,I_N], H^s(\T^2)) \cap Z_N : \|u\|_{Z_N} \leq 2\Lambda \}.
\end{equation*}
When $N >> \frac{1}{s}$, there exists a unique solution $u \in X_N$, depending lipschitz continuously on initial data, to the Duhamel formulation of the modulated cubic nonlinear Schr\"{o}dinger equation \eqref{Eq:ModulatedNLSDuhamel}:
\begin{equation}\label{Eq:Duhamel}
u(t,x) = e^{iW_t\Delta}u_0(x) -i\int_0^{t}e^{i(W_t-W_{t'})\Delta}|u(t',x)|^2u(t',x) \dd t'.
\end{equation}
\end{thm}
\begin{proof}
Because $|w|^2w - |v|^2v = (w-v)(|w|^2+|v|^2) +wv\overline{(w-v)}$, Lemma \ref{Lem:TrilinearEstimate} implies that there exists $\beta > 0$ such that $w \mapsto \mathcal{I}(|w|^2w)$ is a contraction on $X_N$ with lipschitz constant no more than $\frac{1}{2}$. Using (\ref{YsHsEstimate}) we deduce that 
\[
w \mapsto e^{iW_t\Delta}u_0 - i\mathcal{I}(|w|^2w)
\]
is also contraction on $X_N$ with Lipschitz constant no more than $\frac{1}{2}$. It follows that \eqref{Eq:Duhamel} has a solution in $X_N$. If $u_0,v_0 \in H^s$ with $\max(\|u_0\|,\|v_0\|) \leq \frac{\Lambda}{2} $ then using (\ref{YsHsEstimate}),
\begin{equation*}
\|u-v\|_{Z_n} \leq \|e^{iW_t\Delta}(u_0-v_0)\|_{Z_n} + \|\mathcal{I}(|u|^2u-|v|^2v)\|_{Z_N} \leq \|u_0-v_0\|_{H^s} + \frac{1}{2}\|u-v\|_{Z_N},
\end{equation*}
from which it follows that the solution map for Equation \eqref{Eq:Duhamel} is Lipschitz continuous on the the ball $\{v \in H^s: \|v\|_{H^s} \leq \frac{\Lambda}{2}\}$. 

Lastly we prove that solutions in $Z_N$ are unique. For this argument we need to be more precise about the implicit dependencies of the space $Z_N$ and $X_N$ on the parameters $\beta$ and $\Lambda$. let $Z_{N, \beta}$ denote the space $Z_N$, defined in Definition \ref{Defn:Zn}, with parameter $\beta$. Given $\beta $ and $\Lambda$ we write $X_{N, \beta, \Lambda} := \{u \in C^0([0, \beta/N^{4\epsilon}], H^s(\T^2)) \cap Z_{N, \beta}: \|u\|_{Z_{N, \beta}} < 2\Lambda\}$. Suppose that $u,v \in Z_{N, \beta}$ are two solutions to \eqref{Eq:ModulatedNLSDuhamel} on the interval $[0, \beta/N^{4\epsilon}]$. Choose $\Lambda'$ sufficiently large so that $\|u\|_{Z_{N, \beta}} \|v\|_{Z_{N, \beta}} < 2\Lambda'$. As in the construction of local solutions to \eqref{Eq:ModulatedNLSDuhamel} we may choose $\beta' < \beta$ so that $w \mapsto e^{iW_t\Delta}u_0 - i\mathcal{I}(|w|^2w)$ is a contraction on $X_{N, \beta', \Lambda'}$ with Lipschitz constant no more than $\frac{1}{2}$. It follows that $u=v$ on the interval $[0, \beta'/N^{4\epsilon}]$. Because of \eqref{YsLInfinityLowerBoundEstimate} $\|u\|_{L^{\infty}([0, \beta/N^{4\epsilon}], H^s)}, \|v\|_{L^{\infty}([0, \beta/N^{4\epsilon}], H^s)} < 2\Lambda $ so we may iterate this argument to conlcude that $u=v$ on the entire interval $[0, \beta/N^{4\epsilon}]$.
\end{proof}

Applying Theorem \ref{Thm:BrownianIsRhoIrregular} we immediately obtain the following corollary

\begin{cor}\label{Cor:PathWiseLocalWell-posednessBrownian} If $W_t$ is a fractional Brownian motion with Hurst parameter\footnote{When $H=\frac{1}{2}$ this equation is also known as the cubic nonlinear Schr\"{o}dinger equation with white noise dispersion.} $H \leq \frac{1}{2}$ then \eqref{Eq:ModulatedNLSDuhamel} with Duhamel formulation
\begin{equation*}
u(t,x) = e^{iW_t\Delta}u_0(x) -i\int_0^{T}e^{i(W_t-W_{t'})\Delta}|u(t',x)|^2u(t',x) \dd t'
\end{equation*}
is almost surely pathwise locally well posed i.e Theorem \ref{Thm:LocalWellPosedness} applies.
\end{cor}
\section{Additional Results}\label{Sec:AdditionalResults}
In this section we discuss convergence of the linear flow for the propagator  $\{e^{iW_t\Delta}\}_{t \geq 0}$ as well as conservation of mass for \eqref{Eq:ModulatedNLSDuhamel}. Unlike the other sections of this paper we don't restrict ourselves to working in the two dimensional setting. The question of the convergence of the linear Schr\"{o}dinger flow to initial data in the deterministic case was first proposed by Carleson in \cite{carleson1980}. Since then the problem has been studied in numerous different contexts, see \cite{MoyuaVega2008DEquals1SFlow,pointwiseconvergenceschrodingerflow,du2017sharpschrodingermaximalestimate, du2019sharpl2estimateschrodinger} and references therein. Typically, almost everywhere convergence of the linear flow to initial data is not guaranteed to hold for functions that are not sufficiently regular. In both the periodic and nonperiodic settings there are examples of functions lying in the Sobolev space $H^s(\T^d)$ for $s < \frac{d}{2(d+2)}$ where almost everywhere convergence of the linear flow to initial data fails\cite{lucà2015improvednecessaryconditionschrodinger, demeter2016schrodingermaximalfunctionestimates}. Some of the strongest positive results in the nonperiodic setting appear in \cite{du2017sharpschrodingermaximalestimate, du2019sharpl2estimateschrodinger} where it is shown that the linear schr\"{o}dinger flow converges pointwise almost everywhere to initial data when the initial data lies in $H^s(\R^d)$ for $s > \frac{d}{2(d+1)}$. In the periodic case Moyua and Vega showed that the deterministic Schr\"{o}dinger flow converges pointwise almost everywhere to initial data when the initial data lies in $H^s(\T^d)$ for $s > \frac{d}{d+2}$ \cite{MoyuaVega2008DEquals1SFlow}. However, their results are only valid in the case $d=1$. Later this result was extended to the case $d=2$ and finally $d \geq 3$ in \cite{WangZhang2019DEquals2SFlow} and \cite{pointwiseconvergenceschrodingerflow}  respectively. We show that a large class of stochastic processes $W_t$ satisfy a similar property as long as their distribution is sufficiently concentrated near $0$ when $t$ is small. The problem turns out to be trivial when one considers a single modulation function $W_t$ so instead we consider stochastic processes and prove convergence in $L^p(\Omega)$ where the process is defined on the probability space $(\Omega, \FF, P)$. 

\begin{thm}\label{Thm:ConvergenceLinearFlow} 
Let $s > \frac{d}{d+2}$ and $W_t$ be a stochastic process on a probability space $(\Omega, \FF, P)$ with an associated filtration $\FF_t := \sigma\{W_s: 0 \leq s \leq t\}$. Suppose that $1 \leq p < \infty$ and $f(\omega, x) \in  L^p_{\omega}L^{2\frac{d+2}{d}}_x$. Then $e^{iW_t(\omega)\Delta}f(\omega, x) \stackrel{t\to0}{\to}f(\omega, x)$ for almost every $(\omega, x) \in \Omega \times \T^2$ and in $L^p_{\omega}L^{2\frac{d+2}{d}}_x$ provided that
\begin{enumerate}
    \item $W_t \to 0$ as $t \to 0$ almost surely.
    \item $\|\sup_{0 \leq t \leq 1}|W_t|^{\frac{d}{2(d+2)}}\|_{L^p_{\omega}} < \infty$.\label{SecHypoth}
    \item $f(\omega, x)$ is measurable with respect to the $\sigma$-algebra $\FF$ but independent of the $\sigma$-algebra $\FF_{1}$.\label{ThirdHypoth}
\end{enumerate} 
\end{thm}

Moreover, we have the maximal function estimate
\begin{equation}\label{Intermediary30}
\|\sup_{0\leq t \leq 1}|e^{iW_t\Delta} f|\|_{L^p_{\omega}L^{2\frac{d+2}{d}}} \lesssim \|f\|_{L^p_{\omega}H^s_x}.
\end{equation}

\begin{rem}
The first and second hypothesis of Theorem   \ref{Thm:ConvergenceLinearFlow} hold for many commonly studied stochastic processes. In particular they hold for fractional Brownian motion.    
\end{rem}

\begin{proof}
Because $W_t(\omega) \to 0$ pointwise almost surely, convergence of the linear flow for almost every $(\omega, x) \in \Omega \times \T^2$ reduces to the similar problem for the determinisitc Schr\"{o}dinger propagator $\{e^{it\Delta}\}$, which was addressed in \cite[Proposition 3.1]{pointwiseconvergenceschrodingerflow}. In order to show convergence of the linear flow in $L^p_{\omega}L^{2\frac{d+2}{d}}_x$ it is enough to prove \eqref{Intermediary30} and apply the dominated convergence theorem. Using a standard dyadic decomposition we can prove convergence in $L^p_{\omega}L^{2\frac{d+2}{d}}_x$ if we show the maximal estimate
\begin{equation}\label{Intermediary31}
\|\sup_{0\leq t \leq 1}|e^{iW_t\Delta}P_N f|\|_{L^p_{\omega}L^{2\frac{d+2}{d}}} \lesssim N^{\frac{d}{d+2}+}\|f\|_{L^p_{\omega}L^2_x}.
\end{equation}
Here the $+$ in $N^{\frac{d}{d+2}+}$ denotes that for all $\epsilon > 0$ the estimate holds for $N^{\frac{d+2}{d}+\epsilon}$ with an implicit constant depending on $\epsilon$. The idea behind the proof of \eqref{Intermediary31} is to partition $\Omega$ into sets based on the value of $\sup_{0 \leq t \leq 1}|W_t|$. The second hypothesis of this theorem guarantees that this function is small except on a set of small probability. If $\sup_{0 \leq t \leq 1}|W_t|$ is small then we can reduce the problem to the deterministic case by replacing $e^{iW_t\Delta}$ with $e^{i\tau\Delta}$ where $\tau$ ranges over a small interval. Once we reduce to the deterministic case we use the maximal estimate obtained in \cite[Proposition 3.1]{pointwiseconvergenceschrodingerflow}. For convenience we restate the estimate here:
For all $f \in L^2(\T^d)$,
\begin{equation}
\|\sup_{0 \leq t \leq 1}|e^{it\Delta}P_Nf|\|_{L_x^{2\frac{d+2}{d}}} \lesssim N^{\frac{d}{d+2}+}\|f\|_{L^2_x}\label{DeterminsiticMaximalEstimate}
\end{equation}

Let $A_0 := \{\omega\in \Omega: \sup_{0 \leq t \leq 1}|W_t| \in [0,1) \} $ and, for $j \geq 1$, $A_j := \{\omega \in \Omega: \sup_{0 \leq t \leq 1}|W_t| \in [2^{j-1}, 2^{j})\} $. Note that hypothesis \eqref{SecHypoth} guarantees that $\Omega = B \sqcup \bigsqcup_{j=0}^{\infty}A_j$ where $P[B]=0$.  Partitioning $\Omega$ into the $A_j$ sets,
\begin{align}
&\|\sup_{0\leq t \leq 1}|e^{iW_t\Delta}P_N f|\|_{L^p_{\omega}L^{2\frac{d+2} {d}}_x}^p \notag\\
&=
\sum_{j=0}^{\infty} \|1_{A_j}\sup_{0\leq t \leq 1}|e^{iW_t\Delta}P_N f|\|_{L^p_{\omega}L^{2\frac{d+2}{d}}_x}^p \notag\\
&=
\sum_{j=0}^{\infty} \|1_{A_j}\sup_{\{\tau: \tau = W_t \text{ for some }t\in [0,1]\}}|e^{i\tau\Delta}P_Nf|\|_{L^p_{\omega}L^{2\frac{d+2}{d}}_x}^p \notag\\
&\lesssim
\sum_{j=0}^{\infty} \|1_{A_j}\sup_{-2^{j} \leq \tau \leq 2^j}|e^{i\tau\Delta}P_Nf|\|_{L^p_{\omega}L^{2\frac{d+2}{d}}_x}^p \label{Intermediary33}\\
&\lesssim 
N^{\frac{pd}{d+2}+}\sum_{j=1}^{\infty} 2^{\frac{jdp}{2(d+2)}} \|1_{A_j}\|f\|_{L^2_x}\|_{L^p_{\omega}}^p. \label{Intermediary32}
\end{align}
Here the last estimate follows from applying \eqref{DeterminsiticMaximalEstimate} repeatedly on intervals of length $1$ to the $2\frac{d+2}{d}$ power of \eqref{Intermediary33}, taking a $\frac{d}{2(d+2)}$ root, and using that $e^{it\Delta}$ is a unitary operator on $L^2(\T^d)$. Hypothesis \eqref{ThirdHypoth} guarantees that $f$ and $1_{A_j}$ are independent. Using this, we bound \eqref{Intermediary32} by
\begin{align*}
N^{\frac{pd}{d+2}+}\|f\|_{L^p_{\omega}L^2_x}^p \sum_{j=0}^{\infty}2^{\frac{jdp}{2(d+2)}}P[A_j] 
&\lesssim 
N^{\frac{pd}{d+2}+}\|\sup_{0 \leq t \leq 1}|W_t|^{\frac{d}{2(d+2)}}\|^p_{L^p_{\omega}}\|f\|_{L^p_{\omega}L^2_x}^p\\
&\lesssim
N^{\frac{pd}{d+2}+}\|f\|_{L^p_{\omega}L^2_x}^p
\end{align*}
\end{proof}

Conservation of the $L^2$ mass of a local solution to a nonlinear Schr\"{o}dinger equation can be used to control the growth of the solution and extend local results to global results. For example this is used in \cite[Theorem 1.4]{IncidenceGeometryStrichartz} to prove a global well-posedness result for the deterministic cubic NLS. We were unable to extend our local result, Theorem \ref{Thm:LocalWellPosedness}, to a global result, however conservation of mass still holds for solutions to (MNLS).

\begin{prop}\label{Prop:ConservationOfMass}
If $W_t$ is a modulation function and $u(t, x) \in L^{\infty}([0,T], L^2_x)$ satisfies the Duhamel equation
\begin{equation}\label{eq: TruncatedWeakNLSrho}
u(t, x) = e^{iW_t\Delta}u(0, x) -i\int_0^t e^{i(W_t -W_{t'})\Delta}|u(t', x)|^2u(t', x) \dd t', 
\end{equation}
Then the $L^2(\T^2)$ mass, $\|u(t, x)\|_{L^2_{x}}$, is conserved on the interval $[0,T]$.
\end{prop}
\begin{proof}
It is common for the conservation of mass of Schr\"{o}dinger equations to follow from employing an energy method, however this approach is difficult to apply to general modulations $W_t$. Instead, we will differentiate the Fourier coefficients of \eqref{eq: TruncatedWeakNLSrho}. Note that the assumption that $u(t, x) \in L^{\infty}([0,T], L^2_x)$ allows for interchange of the operator $e^{iW_t\Delta}$ and the integral in \eqref{eq: TruncatedWeakNLSrho}.  Since $e^{iW_t\Delta}$ is always a unitary operator on $L^2_x$ it suffices to show that
\begin{align*}
&0 \\
&=
\frac{d}{dt}\bigg{\|}u(0, x) -i\int_0^t e^{-iW_{t'}\Delta}|u(t', x)|^2u(t', x) \dd t'\bigg{\|}_{L^2_x}^2 \\
&= 
\frac{d}{dt} \sum_{k \in \Z^2} \bigg{|} u_k(0) -i\int_0^t e^{iW_{t'}|k|^2} \sum_{k=k_1-k_2+k_3}u_{k_1}(t')\overline{u_{k_2}(t')}u_{k_3}(t')\dd t'\bigg{|}^2\\
&=
I_1+I_2,
\end{align*}
Where
\begin{align*}
&I_1 := \frac{d}{dt}\sum_{k \in \Z^2} \int_0^t\int_0^t e^{i(W_s-W_r)|k|^2}\sum_{\substack{k = k_1-k_2+k_3 \\ k = k_4-k_5+k_6}}u_{k_1}(s)\overline{u_{k_2}(s)}u_{k_3}(s)\\
&\hspace{40ex}\overline{u_{k_4}}(r)u_{k_5}(r)\overline{u_{k_6}(r)} \dd s \dd r \\
& 
I_2 := \frac{d}{dt}\sum_{k \in \Z^2} \int_0^t  \sum_{k=k_1-k_2+k_3}ie^{-iW_{t'}|k|^2}u_k(0)\overline{u_{k_1}}(t')u_{k_2}(t')\overline{u_{k_3}(t')}  \\
&\hspace{11em}
-ie^{iW_{t'}|k|^2}\overline{u_k(0)}u_{k_1}(t')\overline{u_{k_2}(t')}u_{k_3}(t') \dd t'
\end{align*} 
 For almost every $t \in [0,T]$ we have
\begingroup
\allowdisplaybreaks
\begin{align*}
&I_1 =\sum_{k \in \Z^2} \int_0^t e^{i(W_t-W_r)|k|^2} 
\sum_{\substack{k = k_1-k_2+k_3 \\ k = k_4-k_5+k_6}}u_{k_1}(t)\overline{u_{k_2}(t)}u_{k_3}(t)\overline{u_{k_4}(r)}u_{k_5}(r)\overline{u_{k_6}(r)}  \dd r \notag\\
&+
\sum_{k \in \Z^2} \int_0^t e^{i(W_s-W_t)|k|^2} \sum_{\substack{k = k_1-k_2+k_3 \\ k = k_4-k_5+k_6}}u_{k_1}(s)\overline{u_{k_2}(s)}u_{k_3}(s)\overline{u_{k_4}(t)}u_{k_5}(t)\overline{u_{k_6}(t)}  \dd s \notag\\
&=
2\text{RE} \sum_{k \in \Z^2} \int_0^t e^{i(W_t-W_r)|k|^2} \sum_{\substack{k = k_1-k_2+k_3 \\ k = k_4-k_5+k_6}}u_{k_1}(t)\overline{u_{k_2}(t)}u_{k_3}(t)
\overline{u_{k_4}(r)}u_{k_5}(r)\overline{u_{k_6}(r)} \dd r\notag\\
&=
2\text{RE}\sum_{k \in \Z^2} \sum_{k=k_1-k_2+k_3}u_{k_1}(t)\overline{u_{k_2}(t)}u_{k_3}(t)\sum_{k=k_4-k_5+k_6}\int_0^t e^{i(W_t-W_r)|k|^2}\overline{u_{k_4}(r)}u_{k_5}(r)\overline{u_{k_6}(r)} \dd r \notag\\
&=
2\text{RE}\sum_{k \in \Z^2} \sum_{k = k_1-k_2+k_3}u_{k_1}(t)\overline{u_{k_2}(t)}u_{k_3}(t)\overline{\big{(}iu_k(t) - ie^{-iW_t|k|^2}u_k(0) \big{)}} \notag\\
&=
2\text{RE} \int_{\T^2} -i(\overline{u(t,x)}-e^{-iW_t\Delta}\overline{u(0,x)} )|u(t,x)|^2u(t,x) \dd x\notag\\
&=
-2\text{Im} \int_{\T^2}|u(t,x)|^2u(t,x) e^{-iW_t\Delta}\overline{u(0,x)} \dd x,
\end{align*}
\endgroup
and
\begin{align*}
I_2 
&=
\sum_{k \in \Z^2} \sum_{k=k_1-k_2+k_3}ie^{-iW_t|k|^2}u_k(0)\overline{u_{k_1}}(t)u_{k_2}(t)\overline{u_{k_3}(t)} \notag\\
&\hspace{16ex}- ie^{iW_t|k|^2}\overline{u_k(0)}u_{k_1}(t)\overline{u_{k_2}(t)}u_{k_3}(t) \notag\\
&=
2\text{RE} \, i \int_{\T^2}|u(t,x)|^2\overline{u(t,x)}e^{iW_t\Delta}u(0,x) \dd x \notag\\
&=
2\text{Im} \int_{\T^2} |u(t,x)|^2u(t,x)e^{-iW_t\Delta}\overline{u(0,x)} \dd x.
\end{align*}
Therefore $I_1+I_2=0$ as desired.
\end{proof}

\section{Acknowledgment}
The author is grateful to prof. Bjoern Bringmann, prof. Gigliola Staffilani, and Hanzul Munkhbat for many insightful conversations and ideas about different directions to take the project.

\bibliographystyle{plain}
\nocite{*}
\bibliography{bibfile.bib}

@article {bourgain93,
    AUTHOR = {Bourgain, J.},
     TITLE = {Fourier transform restriction phenomena for certain lattice
              subsets and applications to nonlinear evolution equations.
              {I}. {S}chr\"odinger equations},
   JOURNAL = {Geom. Funct. Anal.},
  FJOURNAL = {Geometric and Functional Analysis},
    VOLUME = {3},
      YEAR = {1993},
    NUMBER = {2},
     PAGES = {107--156},
      ISSN = {1016-443X,1420-8970},
   MRCLASS = {35Q55 (11L07 35B10)},
  MRNUMBER = {1209299},
MRREVIEWER = {Yun\ Mei\ Chen},
       DOI = {10.1007/BF01896020},
       URL = {https://doi.org/10.1007/BF01896020},
}

@article {stew,
    AUTHOR = {Stewart, Gavin},
     TITLE = {On the wellposedness of periodic nonlinear {S}chr\"odinger
              equations with white noise dispersion},
   JOURNAL = {Stoch. Partial Differ. Equ. Anal. Comput.},
  FJOURNAL = {Stochastics and Partial Differential Equations. Analysis and
              Computations},
    VOLUME = {12},
      YEAR = {2024},
    NUMBER = {3},
     PAGES = {1417--1438},
      ISSN = {2194-0401,2194-041X},
   MRCLASS = {35R60 (35Q55)},
  MRNUMBER = {4781788},
MRREVIEWER = {Alp\ O.\ Eden},
       DOI = {10.1007/s40072-023-00306-9},
       URL = {https://doi.org/10.1007/s40072-023-00306-9},
}

@article {IncidenceGeometryStrichartz,
    AUTHOR = {Herr, Sebastian and Kwak, Beomjong},
     TITLE = {Strichartz estimates and global well-posedness of the cubic
              {NLS} on {$\Bbb {T}^2$}},
   JOURNAL = {Forum Math. Pi},
  FJOURNAL = {Forum of Mathematics. Pi},
    VOLUME = {12},
      YEAR = {2024},
     PAGES = {Paper No. e14, 21},
      ISSN = {2050-5086},
   MRCLASS = {35Q55 (35A23 52C30)},
  MRNUMBER = {4794808},
MRREVIEWER = {Vedran\ Sohinger},
       DOI = {10.1017/fmp.2024.11},
       URL = {https://doi.org/10.1017/fmp.2024.11},
}

@article {XsYsEstimates,
    AUTHOR = {Herr, Sebastian and Tataru, Daniel and Tzvetkov, Nikolay},
     TITLE = {Global well-posedness of the energy-critical nonlinear
              {S}chr\"odinger equation with small initial data in {$H^1(\Bbb
              T^3)$}},
   JOURNAL = {Duke Math. J.},
  FJOURNAL = {Duke Mathematical Journal},
    VOLUME = {159},
      YEAR = {2011},
    NUMBER = {2},
     PAGES = {329--349},
      ISSN = {0012-7094,1547-7398},
   MRCLASS = {35Q55 (35B30)},
  MRNUMBER = {2824485},
MRREVIEWER = {Matthew\ D.\ Blair},
       DOI = {10.1215/00127094-1415889},
       URL = {https://doi.org/10.1215/00127094-1415889},
}

@article {UpVpEstimates,
    AUTHOR = {Hadac, Martin and Herr, Sebastian and Koch, Herbert},
     TITLE = {Erratum to ``{W}ell-posedness and scattering for the {KP}-{II}
              equation in a critical space'' [{A}nn. {I}. {H}.
              {P}oincar\'e---{AN} 26 (3) (2009) 917--941] [MR2526409]},
   JOURNAL = {Ann. Inst. H. Poincar\'e{} C Anal. Non Lin\'eaire},
  FJOURNAL = {Annales de l'Institut Henri Poincar\'e{} C. Analyse Non
              Lin\'eaire},
    VOLUME = {27},
      YEAR = {2010},
    NUMBER = {3},
     PAGES = {971--972},
      ISSN = {0294-1449,1873-1430},
   MRCLASS = {35Q53 (35B30 35P25)},
  MRNUMBER = {2629889},
       DOI = {10.1016/j.anihpc.2010.01.006},
       URL = {https://doi.org/10.1016/j.anihpc.2010.01.006},
}

@book{UpVptextbook,
  author    = {Herbert Koch and Daniel Tataru and Monica Vi{\c{s}}an},
  title     = {Dispersive Equations and Nonlinear Waves:
               Generalized Korteweg--de Vries, Nonlinear Schr{\"o}dinger,
               Wave and Schr{\"o}dinger Maps},
  series    = {Oberwolfach Seminars},
  volume    = {45},
  publisher = {Birkh{\"a}user},
  year      = {2014},
  doi       = {10.1007/978-3-0348-0736-4},
  isbn      = {978-3-0348-0735-7},
}

@article {chouk2015nonlinearpdesmodulateddispersion,
    AUTHOR = {Chouk, K. and Gubinelli, M.},
     TITLE = {Nonlinear {PDE}s with modulated dispersion {I}: {N}onlinear
              {S}chr\"odinger equations},
   JOURNAL = {Comm. Partial Differential Equations},
  FJOURNAL = {Communications in Partial Differential Equations},
    VOLUME = {40},
      YEAR = {2015},
    NUMBER = {11},
     PAGES = {2047--2081},
      ISSN = {0360-5302,1532-4133},
   MRCLASS = {35Q55 (35R60 60H15 60H99)},
  MRNUMBER = {3418825},
MRREVIEWER = {Clemens\ F.\ Heitzinger},
       DOI = {10.1080/03605302.2015.1073300},
       URL = {https://doi.org/10.1080/03605302.2015.1073300},
}

@article {catellier2016averagingirregularcurves,
    AUTHOR = {Catellier, R. and Gubinelli, M.},
     TITLE = {Averaging along irregular curves and regularisation of {ODE}s},
   JOURNAL = {Stochastic Process. Appl.},
  FJOURNAL = {Stochastic Processes and their Applications},
    VOLUME = {126},
      YEAR = {2016},
    NUMBER = {8},
     PAGES = {2323--2366},
      ISSN = {0304-4149,1879-209X},
   MRCLASS = {34C29 (34F05 60G17 60G22 60H10)},
  MRNUMBER = {3505229},
MRREVIEWER = {Fuke\ Wu},
       DOI = {10.1016/j.spa.2016.02.002},
       URL = {https://doi.org/10.1016/j.spa.2016.02.002},
}

@article {galeati2023prevalencerhoirregularity,
    AUTHOR = {Galeati, Lucio and Gubinelli, Massimiliano},
     TITLE = {Prevalence of {$\rho $}-irregularity and related properties},
   JOURNAL = {Ann. Inst. Henri Poincar\'e{} Probab. Stat.},
  FJOURNAL = {Annales de l'Institut Henri Poincar\'e{} Probabilit\'es et
              Statistiques},
    VOLUME = {60},
      YEAR = {2024},
    NUMBER = {4},
     PAGES = {2415--2467},
      ISSN = {0246-0203,1778-7017},
   MRCLASS = {60H50 (37C99)},
  MRNUMBER = {4828848},
MRREVIEWER = {Jan\ I.\ Seidler},
       DOI = {10.1214/23-aihp1399},
       URL = {https://doi.org/10.1214/23-aihp1399},
}

@misc{herr2025globalwellposedness,
      title={Global well-posedness of the cubic nonlinear Schr\"odinger equation on $\mathbb{T}^{2}$}, 
      author={Sebastian Herr and Beomjong Kwak},
      year={2025},
      eprint={2502.17073},
      archivePrefix={arXiv},
      primaryClass={math.AP},
      url={https://arxiv.org/abs/2502.17073}, 
}

@article {debussche2010quinticNLS,
    AUTHOR = {Debussche, Arnaud and Tsutsumi, Yoshio},
     TITLE = {1{D} quintic nonlinear {S}chr\"odinger equation with white
              noise dispersion},
   JOURNAL = {J. Math. Pures Appl. (9)},
  FJOURNAL = {Journal de Math\'ematiques Pures et Appliqu\'ees. Neuvi\`eme
              S\'erie},
    VOLUME = {96},
      YEAR = {2011},
    NUMBER = {4},
     PAGES = {363--376},
      ISSN = {0021-7824,1776-3371},
   MRCLASS = {35Q55 (35R60 60H15)},
  MRNUMBER = {2832639},
MRREVIEWER = {Mar\'ia\ J.\ Garrido-Atienza},
       DOI = {10.1016/j.matpur.2011.02.002},
       URL = {https://doi.org/10.1016/j.matpur.2011.02.002},
}

@article{DEBOUARD20101300,
title = {The nonlinear Schrödinger equation with white noise dispersion},
journal = {Journal of Functional Analysis},
volume = {259},
number = {5},
pages = {1300-1321},
year = {2010},
issn = {0022-1236},
doi = {https://doi.org/10.1016/j.jfa.2010.04.002},
url = {https://www.sciencedirect.com/science/article/pii/S0022123610001412},
author = {Anne {de Bouard} and Arnaud Debussche}
}

@article {pointwiseconvergenceschrodingerflow,
    AUTHOR = {Compaan, Erin and Luc\`a, Renato and Staffilani, Gigliola},
     TITLE = {Pointwise convergence of the {S}chr\"odinger flow},
   JOURNAL = {Int. Math. Res. Not. IMRN},
  FJOURNAL = {International Mathematics Research Notices. IMRN},
  VOLUME   = {1},
      YEAR = {2021},
    NUMBER = {1},
     PAGES = {599--650},
      ISSN = {1073-7928,1687-0247},
   MRCLASS = {35Q76},
  MRNUMBER = {4198507},
       DOI = {10.1093/imrn/rnaa036},
       URL = {https://doi.org/10.1093/imrn/rnaa036},
}

@article {bourgain2015proofl2decouplingconjecture,
    AUTHOR = {Bourgain, Jean and Demeter, Ciprian},
     TITLE = {The proof of the {$l^2$} decoupling conjecture},
   JOURNAL = {Ann. of Math. (2)},
  FJOURNAL = {Annals of Mathematics. Second Series},
    VOLUME = {182},
      YEAR = {2015},
    NUMBER = {1},
     PAGES = {351--389},
      ISSN = {0003-486X,1939-8980},
   MRCLASS = {42B37 (11E76 46E30 53C40)},
  MRNUMBER = {3374964},
MRREVIEWER = {G.\ V.\ Rozenblum},
       DOI = {10.4007/annals.2015.182.1.9},
       URL = {https://doi.org/10.4007/annals.2015.182.1.9},
}

@book{tao-vu-additive,
  author    = {Terence Tao and Van H. Vu},
  title     = {Additive Combinatorics},
  series    = {Cambridge Studies in Advanced Mathematics},
  volume    = {105},
  year      = {2010},
  publisher = {Cambridge University Press},
  address   = {Cambridge},
  mrnumber  = {2573797}
}

@incollection{carleson1980,
  author       = {Lennart Carleson},
  title        = {Some analytic problems related to statistical mechanics},
  booktitle    = {Euclidean Harmonic Analysis (Proc. Sem., Univ. Maryland, College Park, Md., 1979)},
  series       = {Lecture Notes in Mathematics},
  volume       = {779},
  pages        = {5--45},
  publisher    = {Springer},
  address      = {Berlin},
  year         = {1980}
}

@article {du2019sharpl2estimateschrodinger,
    AUTHOR = {Du, Xiumin and Zhang, Ruixiang},
     TITLE = {Sharp {$L^2$} estimates of the {S}chr\"odinger maximal
              function in higher dimensions},
   JOURNAL = {Ann. of Math. (2)},
  FJOURNAL = {Annals of Mathematics. Second Series},
    VOLUME = {189},
      YEAR = {2019},
    NUMBER = {3},
     PAGES = {837--861},
      ISSN = {0003-486X,1939-8980},
   MRCLASS = {42B20 (42B37)},
  MRNUMBER = {3961084},
MRREVIEWER = {Dong\ Dong},
       DOI = {10.4007/annals.2019.189.3.4},
       URL = {https://doi.org/10.4007/annals.2019.189.3.4},
}

@article {du2017sharpschrodingermaximalestimate,
    AUTHOR = {Du, Xiumin and Guth, Larry and Li, Xiaochun},
     TITLE = {A sharp {S}chr\"odinger maximal estimate in {$\Bbb R^2$}},
   JOURNAL = {Ann. of Math. (2)},
  FJOURNAL = {Annals of Mathematics. Second Series},
    VOLUME = {186},
      YEAR = {2017},
    NUMBER = {2},
     PAGES = {607--640},
      ISSN = {0003-486X,1939-8980},
   MRCLASS = {42B15 (35Q41 42B37)},
  MRNUMBER = {3702674},
MRREVIEWER = {Huoxiong\ Wu},
       DOI = {10.4007/annals.2017.186.2.5},
       URL = {https://doi.org/10.4007/annals.2017.186.2.5},
}

@misc{demeter2016schrodingermaximalfunctionestimates,
      title={Schr\"odinger maximal function estimates via the pseudoconformal transformation}, 
      author={Ciprian Demeter and Shaoming Guo},
      year={2016},
      eprint={1608.07640},
      archivePrefix={arXiv},
      primaryClass={math.CA},
      url={https://arxiv.org/abs/1608.07640}, 
}

@misc{lucà2015improvednecessaryconditionschrodinger,
      title={An improved necessary condition for the Schr\"odinger maximal estimate}, 
      author={Renato Lucà and Keith Rogers},
      year={2015},
      eprint={1506.05325},
      archivePrefix={arXiv},
      primaryClass={math.CA},
      url={https://arxiv.org/abs/1506.05325}, 
}

@article{MoyuaVega2008DEquals1SFlow,
  author    = {A. Moyua and L. Vega},
  title     = {Bounds for the maximal function associated to periodic solutions of one-dimensional dispersive equations},
  journal   = {Bulletin of the London Mathematical Society},
  volume    = {40},
  number    = {1},
  pages     = {117--128},
  year      = {2008},
  doi       = {10.1112/blms/bdm096}
}

@article{WangZhang2019DEquals2SFlow,
  author    = {Xiang Wang and Changjian Zhang},
  title     = {Pointwise convergence of solutions to the Schr\"odinger equation on manifolds},
  journal   = {Canadian Journal of Mathematics},
  volume    = {71},
  number    = {4},
  pages     = {983--995},
  year      = {2019},
  doi       = {10.4153/CJM-2018-029-6}
}

@article{ctx23497407350006761NWienerVarDefn,
author = {Wiener, Norbert},
address = {[Cambridge, Mass.] :},
issn = {0097-1421},
journal = {Journal of mathematics and physics /},
number = {2},
pages = {72-94},
publisher = {The Institute,},
title = {The Quadratic Variation of a Function and its Fourier Coefficients},
volume = {3},
year = {1924-03},
}

@article{KochTataru2005IntroductionOfVariaSpaces,
  author    = {Herbert Koch and Daniel Tataru},
  title     = {Dispersive estimates for principally normal pseudodifferential operators},
  journal   = {Communications on Pure and Applied Mathematics},
  volume    = {58},
  number    = {2},
  pages     = {217--284},
  year      = {2005},
  doi       = {10.1002/cpa.20067}
}
\end{document}